\documentclass[a4paper,10pt]{amsart}

\usepackage[a4paper]{geometry}
\usepackage{overpic}
\usepackage{contour}
\contourlength{1pt}

\newcommand{\la}{\langle}
\newcommand{\ra}{\rangle}
\newcommand{\Z}{\mathbb{Z}}
\newcommand{\R}{\mathbb{R}}
\newcommand{\C}{\mathbb{C}}
\newcommand{\HH}{\mathbb{H}}
\newcommand{\eps}{\varepsilon}
\newcommand{\mz}[2]{%
  \medskip
  \noindent
  \begin{minipage}{.06\linewidth}\hfill$\displaystyle#1$\end{minipage}%
    \hspace{.3em}% 
  \begin{minipage}{.9\linewidth}$\displaystyle#2$\end{minipage}
  }

\DeclareMathOperator{\crr}{\rm cr}
\DeclareMathOperator{\landau}{\rm O}
\DeclareMathOperator{\IM}{\rm Im}
\DeclareMathOperator{\RE}{\rm Re}

\def\dash{\discretionary{-}{}{-}\penalty1000\hskip0pt}

\def\lput(#1,#2)#3{\put(#1,#2){\hbox to 0pt{\hss{#3}}}}
\def\rput(#1,#2)#3{\put(#1,#2){\hbox to 0pt{{#3}\hss}}}
\def\cput(#1,#2)#3{\put(#1,#2){\hbox to 0pt{\hss{#3}\hss}}}

\newcommand{\num}{\renewcommand{\labelenumi}{(\roman{enumi})}
  \renewcommand{\theenumi}{(\roman{enumi})}}

\theoremstyle{plain}
\newtheorem{lem}{Lemma}
\newtheorem{thm}[lem]{Theorem}
\newtheorem{cor}[lem]{Corollary}
\newtheorem{prop}[lem]{Proposition}
\theoremstyle{remark}
\newtheorem{rem}[lem]{Remark}
\theoremstyle{definition}
\newtheorem{defn}[lem]{Definition}

\makeatletter\def\@captionfont{\normalfont\footnotesize}\makeatother

\begin{document}

\author{Christian M{\"u}ller and Amir Vaxman}
\title{Discrete Curvature and Torsion from Cross-Ratios}

\begin{abstract}
  Motivated by a M{\"o}bius invariant subdivision scheme for polygons, we
  study a curvature notion for discrete curves where the cross-ratio plays
  an important role in all our key definitions.
  Using a particular M{\"o}bius invariant point-insertion-rule, comparable
  to the classical four-point-scheme, we construct circles along discrete
  curves. 
  Asymptotic analysis shows that these circles defined on a sampled curve
  converge to the smooth curvature circles as the sampling density
  increases.
  We express our discrete torsion for space curves, which is not a
  M{\"o}bius invariant notion, using the cross-ratio and show its
  asymptotic behavior in analogy to the curvature.
\end{abstract}

\maketitle

\section{Introduction}

Many topics in applied geometry like computer graphics, computer vision,
and geometry processing in general, cover tasks like the acquisition and
analysis of geometric data, its reconstruction, and further its
manipulation and simulation.  Numerically stable approximations of
3D-geometric notions play here a crucial part in creating algorithms that
can handle such tasks. In particular the estimation of \emph{curvatures}
of curves and surfaces is needed in these geometric
algorithms~\cite{boutin-2000,langer+2005,pottmann-2009-iir}.  A good
understanding of estimating curvatures of curves is often an important
step in the direction of estimating curvatures of surfaces.

A different approach to discretized or discrete curvatures than by
numerical approximation comes from discrete differential
geometry~\cite{bobenko+2008}. There the motivation behind any
discretization is to apply the ideas and methods from classical
differential geometry to discrete objects like polygons and meshes without
``simply'' discretizing equations or using classical differential
calculus. Discrete curvature notions of curves are thus connected to a
sensible idea of a curvature circle~\cite{hoffmann-2009} or a consistent
definition of a Frenet\dash frame~\cite{carroll+2014} or geometric ideas
that appear in geometric knot theory~\cite{sullivan-2008}.  Sometimes
discrete definitions of ``differential'' notions of curves are justified
to be sensible by asymptotic analysis and convergence
behavior~\cite{mueller-cosh-2013,sauer-1970}.

In the present paper in a way we combine both strategies. For example, we
show the invariance of our discrete curvature circle with respect to
M{\"o}bius transformations or characterize classes of discrete curves that
are M{\"o}bius equivalent to an arc length parametrization.
On the other hand, and therein lies our focus, we use asymptotic
analysis to justify the definitions of our discrete notions. For example,
in analogy to Sauer~\cite{sauer-1970}, we discretize/sample a smooth curve
$s(t)$ by constructing the inscribed polygon $s(k \eps)$ with $k \in \Z$
as depicted in Figure~\ref{fig:curve} (right). And then we use this
discrete curve to prove, for example, that our discrete curvature
$\kappa_k$, which is defined at the polygon edge $k, k + 1$, is a second order
approximation of the curvature $\kappa$ of $s$, i.e., 
$\kappa_k = \kappa + \landau(\eps^2)$ as $\eps \to 0$, as we will see in
Theorem~\ref{thm:main}. Our definition of $\kappa_k$ will use four
consecutive points as input.
From our definition of the curvature circle we immediately obtain a
discrete Frenet-frame in Theorem~\ref{thm:zweibein} and
Section~\ref{subsec:frenetframe}.

In our definition of the discrete curvature circle appears the so
called cross\dash ratio of four points as main ingredient of its
definition.  The M{\"o}bius invariance of the cross\dash ratio thus
implies the same for the curvature circle, which also holds for smooth
curves. Even our definition for the torsion includes the cross\dash ratio
in its definition, however not only it as the torsion is not M{\"o}bius
invariant.

Our exposition starts with setting the scene in the preliminaries
(Sec.~\ref{sec:prel}). Then we investigate a discrete curvature notion for
planar curves (Sec.~\ref{sec:curvatureplanar}) which we generalize to
space curves in Section~\ref{sec:3dcurves}. In Section~\ref{sec:torsion}
we study a discrete torsion for three\dash dimensional curves.  In
Section~\ref{sec:geometric} we consider some special cases and geometric
properties of a particular `point\dash insertion\dash rule'
(Eqn.~\eqref{eq:inserting}) that plays an important rule in our definition
of the discrete curvature.
Finally, in Section~\ref{sec:num} we perform numerical experiments to
verify our discrete notions of curvature and torsion.

\section{Preliminaries}
\label{sec:prel}

\subsection{Quaternions}
\label{subsec:quaternions}

The Hamiltonian quaternions $\HH$ are very well suited for expressing 
geometry in three dimensional space and in particular for three
dimensional M{\"o}bius geometry (Sec.~\ref{subsec:moebiusgeometry}).
The quaternions constitute a skew field whose elements can be identified
with $\R \times \R^3$. In this paper we write quaternions in the following
way:
$$
  \HH = \{[r, v] \mid r \in \R, v \in \R^3\}.
$$
The first component $r = \RE q$ of a quaternion $q = [r, v]$ is called
\emph{real part} and the second component $v = \IM q$ \emph{imaginary
part}. Consequently, we write $\IM \HH = \{q \in \HH \mid q = [0, v],\
\text{with}\ v \in \R^3\}$.
The addition in this notation of $\HH$ reads $[r, v] + [s, w] = [r + s, v
+ w]$, and the multiplication reads 
$[r, v] \cdot [s, w] = [rs - \la v, w\ra, r w + s v + v \times w]$, 
where $\la \cdot, \cdot \ra$ is the Euclidean scalar product in $\R^3$ and
where $\times$ is the cross product. The \emph{conjugation} of
$q = [r, v]$ is defined by $\overline{q} = [r, -v]$ and the square root of
the real number $q \bar q$ is called \emph{norm} of $q$ and is
denoted by $|q| = \sqrt{q \bar q}$. For every $q \in \HH \setminus \{ 0
\}$ its \emph{inverse} is given by $q^{-1} = \overline{q}/|q|^2$.

Any quaternion $q \in \HH$ can be represented in its \emph{polar
representation} $q = |q| [\cos \phi, v \sin \phi]$ with $\|v\| = 1$ and
$\phi \in [0, \pi]$. In that case we can define the square root of $q$ by
$\sqrt{q} = \sqrt{|q|} [\cos \frac{\phi}{2}, v \sin \frac{\phi}{2}]$.
Also when computing the square root of a complex number we will always
take the principal square root.

Finally, to express points and vectors of $\R^3$ with quaternions we identify
$\R^3$ with $\IM \HH$ via $v \leftrightarrow [0, v]$.

\subsection{M{\"o}bius geometry}
\label{subsec:moebiusgeometry}

A \emph{M{\"o}bius transformation} is a concatenation of a finite number of
reflections $\sigma$ in spheres (center $c$, radius $r$), hence 
$\sigma: \R^n \cup \{\infty\} \to \R^n \cup \{\infty\}$ with $\sigma(x) =
(x - c)/\|x - c\|^2 + c$, $\sigma(\infty) = c$, $\sigma(c) = \infty$.
Invariants in M{\"o}bius geometry are consequently notions and objects
that stay invariant under M{\"o}bius transformations. An important example
of an invariant of planar M{\"o}bius geometry is the complex cross\dash
ratio.

\subsection{Cross-ratio}

The cross\dash ratio is a fundamental notion in geometry, in particular
M{\"o}bius geometry. For four quaternionic numbers $a, b, c, d \in \HH $
the \emph{cross\dash ratio} is defined as
$$
  \crr(a, b, c, d) := (a - b) (b - c)^{-1} (c - d)(d - a)^{-1},
$$
and it is therefore a quaternion itself. The complex numbers $\C$
constitute a subfield in $\HH$. In our notation $\C$ can be embedded in
$\HH$ as $\C \cong \{q \in \HH \mid q = [r, (x, 0, 0)],\ \text{with}\ r, x
\in \R\}$.  Consequently, the cross\dash ratio for complex numbers can be
written in the form
$$
  \crr(a, b, c, d) = \frac{(a - b) (c - d)}{(b - c)(d - a)},
$$
as $\C$ is commutative.

It is well known that the cross\dash ratio of four points in $\R^3$ or in
$\C$ is real if and only if the four points are concyclic (see
e.g.~\cite{bobenko+1996}).

\subsection{Smooth curves}

Our goal is to define a notion of curvature and torsion for discrete
curves (Sec.~\ref{subsec:disccurve}). We will compare our discrete
notions to those of the classical (smooth) differential geometry and as
such to parametrized curves $s: \R \to \R^3$. We will always
assume $s$ to be sufficiently differentiable.  The \emph{curvature}
$\kappa$ and \emph{torsion} $\tau$ of $s$ are given by (see
e.g.~\cite{docarmo-1976})
\begin{equation}
  \label{eq:curvaturetorsion}
  \kappa = \frac{\|s' \times s''\|}{\|s'\|^3},
  \quad\text{and}\quad
  \tau = -\frac{\la s' \times s'', s'''\ra}{\|s' \times s''\|^2}.
\end{equation}
The torsion vanishes if and only if the curve is planar.  For a planar
curve $s: \R \to \R^2$ the curvature is the oriented quantity
\begin{equation}
  \label{eq:curvaturetwod}
\kappa = \frac{\det(s', s'')}{\|s'\|^3}.
\end{equation}

\subsection{Discrete curves}
\label{subsec:disccurve}

By a \emph{discrete curve} we understand a polygonal curve in $\R^2$ or
$\R^3$ which is given by its vertices hence by the map $\gamma: \Z \to
\R^3$. To get visually closer to the notion of a smooth curve we connect
for all $i \in \Z$ consecutive vertices $\gamma(i) \gamma(i + 1)$ by a
straight line segment, the \emph{edges}. However, connecting by line
segments is not crucial in this paper except for better visualizations in
our illustrations. To shorten the notation we will write $\gamma_i$ for
$\gamma(i)$. We call the discrete curve \emph{planar} if it is contained
in a plane, i.e., in a two dimensional affine subspace.

\section{Curvature of Planar Discrete Curves}
\label{sec:curvatureplanar}

We first begin our investigation with a discrete curvature notion for
\emph{planar} curves and extend it in Section~\ref{sec:3dcurves} to curves
in $\R^3$. We identify the two dimensional plane in which our curves live
with the plane of complex numbers $\C$. Before we proceed to the
definition of the curvature (Sec.~\ref{subsec:curvature}) we will consider
a `point\dash insertion\dash rule' in Section~\ref{subsec:inserting}.  We
have also considered this point\dash insertion\dash rule in the context of
a M{\"o}bius invariant subdivision method in~\cite{vaxman+2018}.

\subsection{Point-insertion-rule in $\C$}
\label{subsec:inserting}

Let $a, b, c, d \in \C$ be four pairwise distinct points. We construct a
new point $f(a, b, c, d) \in \C$ in an, at a first glance, very
unintuitive way:
\begin{equation}
  \label{eq:inserting}
  f(a, b, c, d) 
  := 
  \frac{c (b - a) \sqrt{\crr(c, a, b, d)} + b (c - a)}
       {(b - a) \sqrt{\crr(c, a, b, d)} + (c - a)}
       \in \C \cup \infty.
\end{equation}
We will explain more about special cases and the geometric relation of $f$
with respect to $a, b, c, d$ in Section~\ref{sec:geometric}. 

\begin{lem}
  \label{lem:crequation}
  The newly inserted point $f(a, b, c, d)$ fulfills
  $$
    \crr(c, a, b, f(a, b, c, d)) = -\sqrt{\crr(c, a, b, d)}.
  $$
  In particular the construction of $f$ is M{\"o}bius invariant.
\end{lem}
\begin{proof}
  We expand the cross\dash ratio on the left hand side and obtain
  $$
    \frac{(c - a) (b - f)}{(a - b) (f - c)} = -\sqrt{\crr(c, a, b, d)}.
  $$
  Now simple manipulations of this equation yield~\eqref{eq:inserting}.  
  The M{\"o}bius invariance follows immediately, as $f$ can be
  expressed just in terms of cross\dash ratios.
\end{proof}

\begin{figure}[t]
  \begin{overpic}[width=.37\textwidth]{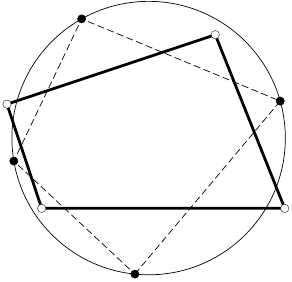}
    \lput(10,21){\small$a$}
    \cput(2,63){\small$b$}
    \lput(71,85){\small$c$}
    \lput(101,19){\small$d$}
    \rput(-4,36){\small$p_{ab}$}
    \lput(26,91){\small$p_{bc}$}
    \rput(96,66){\small$p_{cd}$}
    \cput(44,-2){\small$p_{da}$}
  \end{overpic}
  \hfill
  \begin{overpic}[width=.36\textwidth]{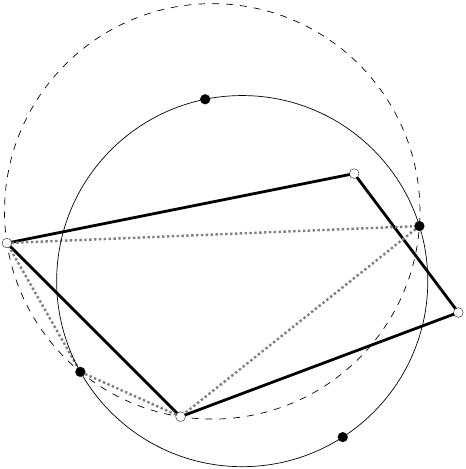}
    \lput(38,6){\small$b$}
    \lput(0,43){\small$a$}
    \lput(76,65){\small$d$}
    \rput(97,28){\small$c$}
    \lput(16,17){\small$p_{ab}$}
    \rput(91,55){\small$p_{cd}$}
    \cput(37,80){\small$p_{da}$}
    \cput(80,3){\small$p_{bc}$}
    \cput(20,70){\small$k$}
  \end{overpic}
  \hfill
  \begin{overpic}[width=.17\textwidth]{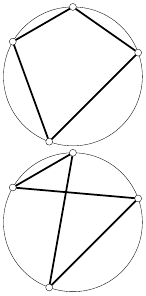}
    \lput(15,49){\small$a$}
    \lput(49,84){\small$b$}
    \lput(26,100){\small$c$}
    \rput(0,88){\small$d$}
    \lput(14,0){\small$a$}
    \lput(30,45){\contour{white}{\small$b$}}
    \lput(2,35){\small$c$}
    \rput(49,29){\small$d$}
  \end{overpic}
  \caption{\emph{Left}: A quadrilateral $a, b, c, d \in \C$ with its newly
  inserted points $p_{ab}, p_{bc}, p_{cd}, p_{da}$ which have a cross\dash
  ratio of $-1$ and therefore lie on a common circle. 
    \emph{Center}: The four points $a, p_{ab}, b, p_{cd}$ also have a
    cross\dash ratio of $-1$ and lie therefore also on a common circle $k$.
    Furthermore, $\crr(a, p_{ab}, b, p_{cd}) = -1$ implies that the pair
    $(a, b)$ is separated by the pair $(p_{ab}, p_{cd})$. Consequently,
    $a$ and $b$ lie on different sides of $k$.
    \emph{Right}: Two concyclic quadrilaterals, convex (\emph{top}) and
    non\dash convex with crossing edges (\emph{bottom}).
    }
  \label{fig:quadpluscircle}
\end{figure}
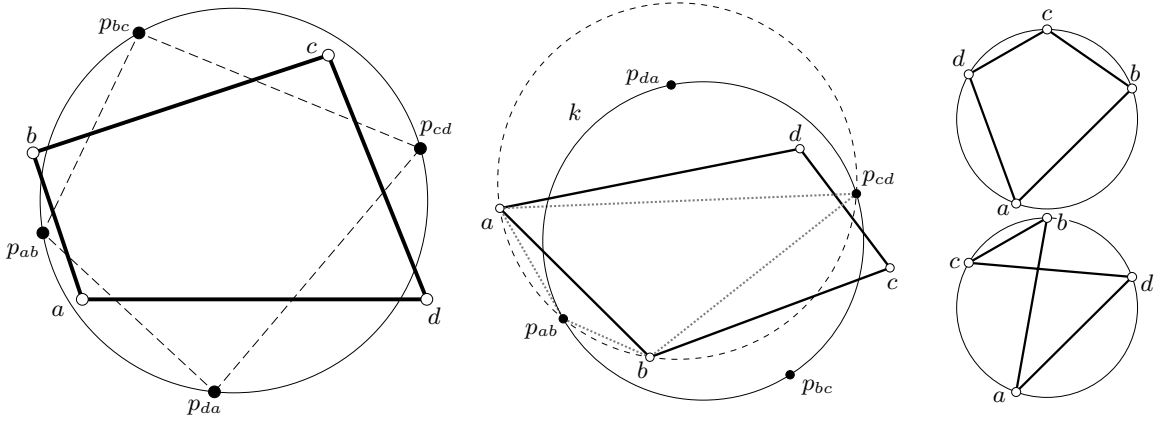

\begin{thm}
  \label{thm:circle}
  Let $a, b, c, d \in \C$ be four pairwise distinct points and consider
  the four new points obtained from $f$ by cyclic permutation:
  $$
    p_{ab} = f(d, a, b, c),\
    p_{bc} = f(a, b, c, d),\
    p_{cd} = f(b, c, d, a),\
    p_{da} = f(c, d, a, b).
  $$
  Then $p_{ab}, p_{bc}, p_{cd}, p_{da}$ are concyclic with $\crr(p_{ab},
  p_{bc}, p_{cd}, p_{da}) = -1$ (see Figure~\ref{fig:quadpluscircle} left).
\end{thm}
\begin{proof}
  It is a well known fact that the cross\dash ratio of four points is real if
  and only if the four points lie on a circle. Hence, we only have to show
  the second part namely
  \begin{equation}
    \label{eq:crvier}
    \frac{(p_{ab} - p_{bc}) (p_{cd} - p_{da})}{(p_{bc} - p_{cd})(p_{da} -
    p_{ab})} = -1,
    \quad\text{or equivalently}\quad
    2 p_{ab} p_{cd} + 2 p_{bc} p_{da} = 
    (p_{ab} + p_{cd})(p_{bc} + p_{da}).
  \end{equation}
  Another well known (and readily verifiable) fact about the cross\dash
  ratio is $\crr(b, a, d, c) = \crr(a, b, c, d)$. Consequently, the
  cross\dash ratios that appear in the definition of $p_{ab}$ and $p_{cd}$
  are the same as well as in $p_{bc}$ and $p_{da}$.
  So, let us denote by $q$ the cross\dash ratios $q := \crr(c, a, b, d) =
  \crr(a, c, d, b)$ and let us start to collect the terms of the second
  equation of
  \eqref{eq:crvier}:
  \begin{align*}
    p_{bc} p_{da}
    &=
    \frac{c (b - a) \sqrt{q} + b (c - a)}
       {(b - a) \sqrt{q} + (c - a)}
       \cdot
    \frac{a (d - c) \sqrt{q} + d (a - c)}
       {(d - c) \sqrt{q} + (a - c)}
    = \ldots =
    \\
    &=
    \frac{(1 + \sqrt{q}) (a b d - a b c - b c d + a c d)}{(1 + \sqrt{q})
    (a - b - c + d)}
    =
    \frac{a b d - a b c - b c d + a c d}{a - b - c + d}.
  \end{align*}
  As $p_{ab}$ and $p_{cd}$ result from $p_{da}$ and $p_{bc}$,
  respectively, by a cyclic permutation of one step ($a \to b, b \to c, c
  \to d, d \to a$) we immediately obtain by shifting from the last identity 
  \begin{equation}
    \label{eq:abcd}
    p_{ab} p_{cd}
    =
    \frac{a b c - b c d - a c d + a b d}{a + b - c - d}.
  \end{equation}
  Next we compute the factors of the right hand side:
  \begin{align*}
    p_{bc} + p_{da}
    =
    \frac{c (b - a) \sqrt{q} + b (c - a)}
       {(b - a) \sqrt{q} + (c - a)}
    +
    \frac{a (d - c) \sqrt{q} + d (a - c)}
       {(d - c) \sqrt{q} + (a - c)}
    = \ldots =
    \frac{2 a d - 2 b c}{a - b - c + d},
  \end{align*}
  and again the same permutation of one shift yields
  \begin{equation}
    \label{eq:abpcd}
    p_{ab} + p_{cd}
    = 
    \frac{2 a b - 2 c d}{a + b - c - d}.
  \end{equation}
  Adding and multiplying these notions together yields
  Equation~\eqref{eq:crvier}.
\end{proof}

Lemma~\ref{lem:crequation} immediately implies the following two important
consequences:
\begin{cor}
  \label{cor:moebinv}
  \begin{enumerate}\num
    \item If $a, b, c, d$ lie on a circle such that $\crr(c, a, b, d) > 0$
      (which is the case for a convex quadrilateral, i.e., non\dash crossing
      edges; see Figure~\ref{fig:quadpluscircle} right) then the four points
      $p_{ab}, p_{bc}, p_{cd}, p_{da}$ lie on the same circle.
    \item\label{cor:moebinvii} The circle given by
      Theorem~\ref{thm:circle} is connected to $a,
      b, c, d$ in a M{\"o}bius invariant way.
  \end{enumerate}
\end{cor}

\begin{lem}
  \label{lem:harmonic}
  For any four pairwise distinct points $a, b, c, d \in \C$ the
  harmonic conjugate of $p_{ab}$ with respect to $a, b$ is $p_{cd}$, which
  equivalently means 
  $\crr(a, p_{ab}, b, p_{cd}) = -1$ 
  (see also Figure~\ref{fig:quadpluscircle} center).  Analogously, for the
  other quadruples we have
  $\crr(b, p_{bc}, c, p_{da}) = -1$,
  $\crr(c, p_{cd}, d, p_{ab}) = -1$, and
  $\crr(d, p_{da}, a, p_{bc}) = -1$.
\end{lem}
\begin{proof}
  We show 
  $$
    \crr(a, p_{ab}, b, p_{cd}) = -1,
    \quad\text{or equivalently}\quad
    2 a b + 2 p_{ab} p_{cd} = (p_{ab} + p_{cd}) (a + b).
  $$
  The product $p_{ab} p_{cd}$ has been computed before in
  Equation~\eqref{eq:abcd} and the sum $p_{ab} + p_{cd}$ in
  Equation~\eqref{eq:abpcd}. Multiplying these terms together as written
  above on the right hand side concludes the proof.
\end{proof}

\begin{cor}
  \label{cor:insideoutside}
  Let $a, b, c, d \in \C$ be four pairwise distinct points and let $k$
  denote the circle through $p_{ab}, p_{bc}, p_{cd}, p_{da}$. Then either
  all eight points lie on the same circle $k$, or $a, c$ lie on one side of
  $k$ and $b, d$ on the other side (see Figure~\ref{fig:quadpluscircle}
  center).
\end{cor}
\begin{proof}
  The two points $p_{ab}$ and $p_{cd}$ lie on the circle $k$.
  Suppose $a$ lies outside of $k$ as in Figure~\ref{fig:quadpluscircle}
  (center).
  Then Lemma~\ref{lem:harmonic} implies that $b$ lies on a circle through
  $a, p_{ab}, p_{cd}$, and further, that $b$ is separated from $a$ by
  $p_{ab}$ and $p_{cd}$.
  Consequently, $b$ lies inside $k$ (see Figure~\ref{fig:quadpluscircle}
  center). The same argument then implies that $c$ lies outside again and
  further $d$ inside. 
\end{proof}

\begin{figure}[t]
  \begin{overpic}[width=.5\textwidth]{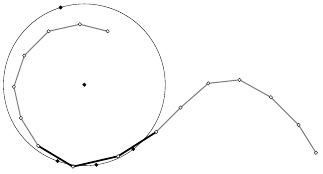}
    \rput(12,10){\tiny$\gamma_{i - 1}$}
    \rput(21,4){\tiny$\gamma_{i}$}
    \lput(37,7){\tiny$\gamma_{i + 1}$}
    \rput(49,11){\tiny$\gamma_{i + 2}$}
    \lput(18,2){\tiny$p_{ab}$}
    \rput(29,0){\tiny$p_{bc}$}
    \rput(40,5){\tiny$p_{cd}$}
    \lput(18,52){\tiny$p_{da}$}
    \lput(51,42){\small$k_i$}
    \put(28,25){\small$m_i$}
  \end{overpic}
  \hfill
  \begin{overpic}[width=.45\textwidth]{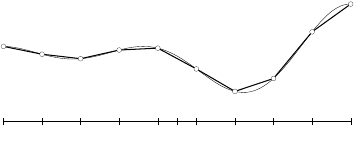}
    \cput(-1,2){\tiny$u\!\! +\!\! (2k\!\! -\!\! 1) \eps$}
    \cput(12.5,2){\tiny$u\!\! -\!\! 7\eps$}
    \cput(23.5,2){\tiny$u\!\! -\!\! 5\eps$}
    \cput(34.5,2){\tiny$u\!\! -\!\! 3\eps$}
    \cput(44.5,2){\tiny$u\!\! -\!\! \eps$}
    \cput(50,2){\tiny$u$}
    \cput(55.5,2){\tiny$u\!\! +\!\! \eps$}
    \cput(67,2){\tiny$u\!\! +\!\! 3\eps$}
    \cput(78,2){\tiny$u\!\! +\!\! 5\eps$}
    \cput(89,2){\tiny$u\!\! +\!\! 7\eps$}
    \rput(0,9){\tiny$k$}
    \cput(12.5,9){\tiny$-3$}
    \cput(23.5,9){\tiny$-2$}
    \cput(34.5,9){\tiny$-1$}
    \cput(44.5,9){\tiny$0$}
    \cput(55.5,9){\tiny$1$}
    \cput(67,9){\tiny$2$}
    \cput(78,9){\tiny$3$}
    \cput(89,9){\tiny$4$}
    \cput(12.5,22){\tiny$\gamma_{-3}$}
    \cput(23.5,20){\tiny$\gamma_{-2}$}
    \cput(34.5,23){\tiny$\gamma_{-1}$}
    \cput(44.5,23){\tiny$\gamma_{0}$}
    \cput(55.5,18){\tiny$\gamma_{1}$}
    \cput(67.0,18){\tiny$\gamma_{2}$}
    \cput(76.0,21){\tiny$\gamma_{3}$}
    \cput(89.0,28){\tiny$\gamma_{4}$}
  \end{overpic}
  \caption{\emph{Left}: A planar discrete curve $\gamma: \Z \to \C$ with the 
  discrete curvature circle $k_i$ at edge $\gamma_i \gamma_{i + 1}$ (the points
  $p_{ab}$ correspond to $p_{\gamma_{i - 1} \gamma_{i}}$ etc). 
    \emph{Right}: Sampling a smooth curve $s: \R \to \R^2$ at $u + (2 k -
    1) \eps$ to obtain the discrete curve $\gamma: \Z \to \R^2$ with
    $\gamma_k = s(u + (2 k - 1) \eps)$. For our asymptotic analysis we let
    the real number $\eps$ go to zero.
    }
  \label{fig:curve}
\end{figure}

\subsection{Curvature for planar curves}
\label{subsec:curvature}

Let us consider the planar discrete curve $\gamma: \Z \to \C$ as
illustrated in Figure~\ref{fig:curve} (left).  We assume that any four
consecutive vertices of the curve are pairwise distinct. Then,
Theorem~\ref{thm:circle} guarantees the existence of a circle $k_i$
passing through 
$f(\gamma_{i - 1}, \gamma_{i}, \gamma_{i + 1}, \gamma_{i + 2})$,
$f(\gamma_{i}, \gamma_{i + 1}, \gamma_{i + 2}, \gamma_{i - 1})$,
$f(\gamma_{i + 1}, \gamma_{i + 2}, \gamma_{i - 1}, \gamma_{i})$,
$f(\gamma_{i + 2}, \gamma_{i - 1}, \gamma_{i}, \gamma_{i + 1})$.
We use this circle $k_i$ in the following definition of our discrete
curvature.

\begin{defn}
  Let $\gamma: \Z \to \C$ be a planar discrete curve. We call the circle
  $k_i$ \emph{(discrete) curvature circle at the edge $\gamma_{i}\gamma_{i
  + 1}$}, the inverse of its radius \emph{(discrete) curvature $\kappa_i$
  at the edge $\gamma_{i}\gamma_{i + 1}$}, and its center $m_i$
  \emph{(discrete) curvature center}. For an illustration see
  Figure~\ref{fig:curve} (left). 
\end{defn}

A `good' discrete definition `mimics' its smooth counterparts. Along these
lines we note that our discrete curvature circle is M{\"o}bius invariant
(Corollary~\ref{cor:moebinv}~\ref{cor:moebinvii}) as in the smooth case.
Furthermore, the curvature circle of a discrete curve with vertices on a
circle -- we could call it a discrete circle -- is the circumcircle
itself, as expected.  And Corollary~\ref{cor:insideoutside} implies that
the curvature circle separates the first and the last point of the four
points that are involved in its definition (see Figure~\ref{fig:curve}
left).  This resembles the local behavior of smooth curves in non\dash
vertex points where the curvature circle separates locally the curve into
an `inner' and an `outer' curve. 

In the following we continue our argumentation for the reasonableness of
this definition of the discrete curvature circle with \emph{asymptotic
analysis}. We will show that the discrete curvature circle (its radius and
center) of a sampled curve $s$ converges to the smooth curvature circle as
the sampling gets denser and denser.  For an illustration of the setting
of the following theorem see Figures~\ref{fig:curve} and~\ref{fig:asymptotic}.

\begin{figure}[t]
  \begin{overpic}[width=.9\textwidth]{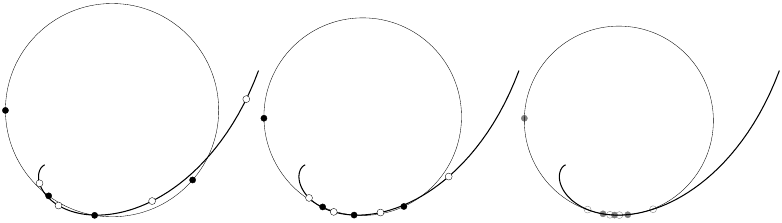}
  \end{overpic}
  \caption{The white points mark four sampled points which move closer and
  closer to a common point from left to right. The associated discrete
  curvature circle (exactly passing through the black points) converges at
  the same time to the smooth curvature circle.}
  \label{fig:asymptotic}
\end{figure}

\begin{thm}
  \label{thm:main}
  Let $s: \R \to \C$ be a sufficiently smooth planar curve and let $u,
  \eps \in \R$. Let further $\gamma: \Z \to \C$ be the planar discrete curve
  that samples the smooth curve $s$ in the following way:
  $$
    \gamma_k = \gamma(k) := s(u + (2 k - 1)\eps) 
    \qquad k \in \Z.
  $$ 
  Then the \emph{discrete} curvature $\kappa_0$ of $\gamma$ at the edge
  $\gamma_0 \gamma_1$ is a second order approximation of the \emph{smooth}
  curvature $\kappa$ of $s$ at $u$:
  \begin{equation}
    \label{eq:curvature}
    \kappa_0 = |\kappa(u)| + \landau(\eps^2).
  \end{equation}
  The center $m_0$ of the \emph{discrete} curvature circle $k_0$ of
  $\gamma$ converges to the center of the \emph{smooth} curvature circle
  of $s$ at the same rate,
  \begin{equation}
    \label{eq:maincenter}
      m_0 = s(u) + \frac{1}{\kappa(u)} N(u) +
    \landau(\eps^2),
  \end{equation}
  where $N$ denotes the unit normal vector along $s$.
  Furthermore, $p_{\gamma_0 \gamma_1} = f(\gamma_{-1}, \gamma_{0},
  \gamma_{1}, \gamma_{2})$ is even a third order approximation of $s(u)$,
  i.e.,
  \begin{equation}
    \label{eq:mainpkt}
    p_{\gamma_0 \gamma_1} = s(u) + \landau(\eps^3).
  \end{equation}
\end{thm}

Before we give a proof of this theorem we need a couple of preparatory
lemmas.  We consider w.l.o.g.\ the approximation point at $u = 0$.
To study the asymptotic behavior of the curvature notions based on our
sampled curve we need its Taylor expansion at $0$:
$$
  s(u) 
  = 
  s(0) + u s'(0) + \frac{u^2}{2} s''(0) + \frac{u^3}{6} s'''(0)
  + \landau(u^4).
$$
For the sake of brevity we will just write $s$ instead of $s(0)$, $s'$
instead of $s'(0)$, etc.
And until the end of this section we will use the following abbreviations
for those four points on which the curvature circle depends:
\begin{equation}
  \label{eq:epspkt}
  a := \gamma_{-1} = s(-3 \eps),
  \quad
  b := \gamma_0 = s(-\eps),
  \quad
  c := \gamma_1 = s(\eps),
  \quad
  d := \gamma_2 = s(3 \eps).
\end{equation}

We will very frequently encounter rational functions depending on $\eps$
for which we need its Taylor expansion. So at first, a general technical
lemma that can easily be verified.
\begin{lem}
  \label{lem:hilfslemmataylor}
  Let $x_i, y_i \in \C$ with $y_0 \neq 0$, then
  $$
    \frac{\sum_{k = 0}^2 x_i \eps^i}{\sum_{k = 0}^2 y_i \eps^i}
    =
    \frac{x_0}{y_0}
    + \frac{x_1 y_0 - x_0 y_1}{y_0^2} \eps
    + \frac{x_2 y_0^2 - x_1 y_0 y_1 + x_0 y_1^2 - x_0 y_0 y_2}{y_0^3} \eps^2
    + \landau(\eps^3).
  $$
\end{lem}

Our first task is to compute the Taylor expansion of the cross\dash ratio
which appears in the `inserting' construction~\eqref{eq:inserting}. The
following formula illustrates also the close connection between the
cross\dash ratio and the \emph{Schwarzian derivative} of $s$ which reads
$\frac{2 s' s''' - 3 f''^2}{2 f'^2}$, cf.~\cite{buecking-2018}.
\begin{lem}
  \label{lem:hilfslemmaeins}
  Let $a, b, c, d$ be the four consecutive points of the sampled curve as
  defined in~\eqref{eq:epspkt}. Then
  $$
    \crr(c, a, b, d) 
    = 
    4 + \frac{12 s''^2 - 8 s' s'''}{s'^2} \eps^2 + \landau(\eps^4),
    \quad
    \crr(b, d, a, c)
    = 
    \frac{4}{3} + \frac{-12 s''^2 + 8 s' s'''}{9 s'^2} \eps^2 + \landau(\eps^4).
  $$
\end{lem}
\begin{proof}
  We start by computing the factors of 
  $$
    \crr(c, a, b, d) = \frac{(c - a) (b - d)}{(a - b) (d - c)}
  $$
  in terms of the Taylor expansion:
  \begin{align*}
    c - a 
    &= 
    s(\eps) - s(-3 \eps)
    =
    \eps s' + \frac{\eps^2}{2} s'' + \frac{\eps^3}{6} s''' + \landau(\eps^4)
    - \big(-3 \eps s' + \frac{9 \eps^2}{2} s'' - \frac{9 \eps^3}{2} s''' +
    \landau(\eps^4)\big)
    \\
    &=
    4 \eps s' - 4 \eps^2 s'' + \frac{14}{3} \eps^3 s''' + \landau(\eps^4).
  \end{align*}
  And analogously we obtain
  \begin{align}
    a - b 
    &=
    -2 \eps s' + 4 \eps^2 s'' - \frac{13}{3} \eps^3 s''' +
    \landau(\eps^4),
    \label{eq:amb}
    \\ \nonumber
    b - d 
    &=
    -4 \eps s' - 4 \eps^2 s'' - \frac{14}{3} \eps^3 s''' +
    \landau(\eps^4),
    \\ \nonumber
    d - c 
    &=
    2 \eps s' + 4 \eps^2 s'' + \frac{13}{3} \eps^3 s''' +
    \landau(\eps^4).
  \end{align}
  Now the numerator of the cross\dash ratio expands to 
  $$
    (c - a) (b - d)
    =
    -16 s'^2 \eps^2 
    + \Big(
    16 s''^2 - \frac{112 s' s'''}{3} 
    \Big)
    \eps^4 
    + \landau(\eps^6),
  $$
  and the denominator to
  $$
    (a - b) (d - c)
    =
    -4 s'^2 \eps^2 
    + \Big(
    16 s''^2 - \frac{52 s' s'''}{3} 
    \Big)
    \eps^4 
    + \landau(\eps^6).
  $$
  Consequently, after canceling $-4 \eps^2$ the cross\dash ratio reads
  $$
    \crr(c, a, b, d) 
    = 
    \frac{4 s'^2 - (4 s''^2 - \frac{28}{3} s' s''') \eps^2 +
    \landau(\eps^4)}
    {s'^2 - (4 s''^2 - \frac{13}{3} s' s''') \eps^2 +
    \landau(\eps^4)},
  $$
  which, using Lemma~\ref{lem:hilfslemmataylor}, simplifies to 
  $$
    \crr(c, a, b, d) 
    = 
    4
    +
    \frac{12 s''^2 - 8 s' s'''}{s'^2} \eps^2 
    +
    \landau(\eps^4),
  $$
  the Taylor expansion of the first cross\dash ratio. The computations for
  the second one work analogously.
\end{proof}

\begin{lem}
  \label{lem:sqrtcr}
  Let $a, b, c, d$ be as in Lemma~\ref{lem:hilfslemmaeins}. Then
  $$
    \sqrt{\crr(c, a, b, d)} 
    = 
    2 + \frac{3 s''^2 - 2 s' s'''}{s'^2} \eps^2 + \landau(\eps^3).
  $$
\end{lem}
\begin{proof}
  This equation follows immediately from the general Taylor expansion for 
  $$
    \sqrt{x_0 + x_2 \eps^2 + \landau(\eps^3)}
    =
    \sqrt{x_0} + \frac{x_2}{2 \sqrt{x_0}} \eps^2 + \landau(\eps^3),
  $$
  and from Lemma~\ref{lem:hilfslemmaeins}.
\end{proof}

Now we are in the position to show the important lemma that guarantees that
$p_{\gamma_0 \gamma_1} = p_{bc}$ is a third order approximation of $s$.

\begin{lem}
  \label{lem:pbc}
  Let $a, b, c, d$ be as in Lemma~\ref{lem:hilfslemmaeins}. Then
  $$
    p_{bc} = s + \landau(\eps^3).
  $$
\end{lem}
\begin{proof}
  We have to compute
  $$
    p_{bc} = f(a, b, c, d) =
    \frac{c (b - a) \sqrt{\crr(c, a, b, d)} + b (c - a)}
         {(b - a) \sqrt{\crr(c, a, b, d)} + (c - a)},
  $$
  and start with its components:
  \begin{align*}
    c (b - a)
    &\overset{\eqref{eq:amb}}{=}
    \big(
    s + \eps s' + \frac{\eps^2}{2} s'' + \frac{\eps^3}{6} +
    \landau(\eps^4)
    \big)
    \big(
    2 \eps s' - 4 \eps^2 s'' + \frac{13}{3} \eps^3 s''' +
    \landau(\eps^4)
    \big)
    \\
    &=
    2 \eps s s' 
    + (2 s'^2 - 4 s s'') \eps^2 
    + (\frac{13}{3} s' s''' - 3 s' s'') \eps^3 
    + \landau(\eps^4),
  \end{align*}
  and analogously
  \begin{equation*}
    b (c - a)
    =
    4 s s' \eps - 4 (s'^2 + s s'') \eps^2
    + (6 s' s'' + \frac{14 s s'''}{3}) \eps^3
    + \landau(\eps^4).
  \end{equation*}
  Putting numerator and denominator together also using
  Lemma~\ref{lem:sqrtcr} we obtain 
  $$
    p_{bc} 
    = 
    \frac{8 s s' \eps 
    - 12 s s'' \eps^2 
    + \big(
      \frac{6 s''^2}{s'}
      +
      \frac{28 s'''}{3}
    \big) s \eps^3
    + \landau(\eps^4)}
    {8 s' \eps 
    - 12 s'' \eps^2 
    + \big(
      \frac{6 s''^2}{s'}
      +
      \frac{28 s'''}{3}
    \big) \eps^3
    + \landau(\eps^4)},
  $$
  and after canceling $\eps$, applying Lemma~\ref{lem:hilfslemmataylor}
  concludes the  proof.
\end{proof}

Note that Lemma~\ref{lem:pbc} proves Equation~\eqref{eq:mainpkt} in
Theorem~\ref{thm:main} which says that $= p_{bc}$ converges to $s$ at
third order.  The following lemma can be verified analogously to
Lemma~\ref{lem:pbc}.

\begin{lem}
  \label{lem:pda}
  Let $a, b, c, d$ be as in~\eqref{eq:epspkt}. Then
  \begin{align*}
    p_{ab} 
    &= 
    s - \sqrt{3} s' \eps + \frac{3 s''}{2} \eps^2 + \landau(\eps^3),
    \\
    p_{cd} 
    &= 
    s + \sqrt{3} s' \eps + \frac{3 s''}{2} \eps^2 + \landau(\eps^3),
    \\
    p_{da} 
    &= 
    s - \frac{2 s'^2}{s''} + 
    \Big(
      5 s'' - \frac{20 s' s'''}{3 s''}
    \Big) \eps^2
    + \landau(\eps^3).
  \end{align*}
\end{lem}

Lemma~\ref{lem:pda} implies that $p_{da} = p_{\gamma_{i + 2} \gamma_{i -
1}}$ is a second order approximation of $s - \frac{2 s'^2}{s''}$. This
point $s - \frac{2 s'^2}{s''}$ has an interesting geometric interpretation
which we detail in Proposition~\ref{prop:krkrpkt}.

After putting these preparatory lemmas in place we can finally turn to the
proof of our result on the limit of the curvature circle.

\begin{proof}[Proof of Theorem \ref{thm:main}]
  Let us first compute the center $m_0$ of the discrete curvature
  circle $k_0$. Generally, the circumcenter of a triangle $a, b, c \in
  \C^2$ is given by
  $$
  \frac{
    a (\|b\|^2 - \|c\|^2)
    +
    b (\|c\|^2 - \|a\|^2)
    +
    c (\|a\|^2 - \|b\|^2)
  }{
  (a - c)
  (\overline{b - c}) 
  -
  (\overline{a - c})
  (b - c) 
  }.
  $$
  In our case we want to compute the circumcenter of the four concyclic
  points $p_{ab}, p_{bc}, p_{cd}, p_{da}$ from which we choose the three
  points $p_{bc}, p_{cd}, p_{da}$ to insert them in the formula above. Let us
  start with the denominator $\mathcal{D}$: 

  \mz{\mathcal{D} =}{
      (p_{da} - p_{cd}) (\overline{p_{bc} - p_{cd}})
      -
      (\overline{p_{da} - p_{cd}}) (p_{bc} - p_{cd})
  }

  \mz{=}{
    \big(-\frac{2 s'^2}{s''} - \sqrt{3} s' \eps + \landau(\eps^2)\big)
    \big(-\sqrt{3} \bar s' \eps - \frac{3 \bar s''}{2} \eps^2 
    + \landau(\eps^3)\big)
    -
    \overline{(\ldots)}(\ldots)
  }

  \mz{=}{
    \frac{2 \sqrt{3} |s'|^2 (s' \bar s'' - \bar s' s'')}{|s''|^2} \eps
    +
    \frac{3 (s' \bar s'' - \bar s' s'') (s' \bar s'' + \bar s' s'')}
    {|s''|^2} \eps^2 + \landau(\eps^3).
  }
  
  We compute the numerator $\mathcal{N}$ in the same way. After a lengthy 
  computation we get

  \mz{\mathcal{N} =}{
    \frac{4 \sqrt{3} |s'|^2 (|s'|^2 s' - \frac{1}{2} s (s'' \bar s' - \bar s''
    s'))}{|s''|^2} \eps
    +
    \frac{6 (s' \bar s'' \!+\! \bar s' s'') (|s'|^2 s' - \frac{1}{2} s (s''
    \bar s' - \bar s''
    s'))}{|s''|^2} \eps^2 + \landau(\eps^3).
  }

  Now, Lemma~\ref{lem:hilfslemmataylor} yields for the center $m_0$ of the
  discrete curvature circle $k_0$
  \begin{equation}
    \label{eq:nd}
    m_0
    =
    \frac{\mathcal{N}}{\mathcal{D}} 
    =
    s + \frac{2 s'^2 \bar s'}{s' \bar s'' - \bar s' s''} +
    \landau(\eps^2).
  \end{equation}
  We need to relate the discrete curvature circle to its smooth
  counterpart. In order to do that, we rewrite the
  curvature~\eqref{eq:curvaturetwod} in terms of complex functions: 
  The determinant of a matrix consisting of two column vectors $a, b \in
  \R^2$ is the same as $\frac{i}{2} (a \bar b - \bar a b)$ when $a$ and
  $b$ are expressed as complex numbers. Consequently, the curvature for a
  curve $s: \R \to \C$ and its unit normal vector $N$ can be written in
  the form 
  \begin{equation}
    \label{eq:detcurvature}
    \kappa = \frac{i (s' \bar s'' - \bar s' s'')}{2 |s'|^3},
		\quad\text{and}\quad
    N = i \frac{s'}{|s'|},
  \end{equation}
  as multiplication with $i$ corresponds to a rotation about the angle
  $\pi/2$.
  So, we use these notions to rewrite~\eqref{eq:nd}:
  $$
    m_0
    =
    s + \frac{|s'|^2 s' i}{\frac{i}{2}(s' \bar s'' - \bar s' s'')} 
    + \landau(\eps^2)
    =
    s + \frac{|s'|^3}{\det(s', s'')} 
    i \frac{s'}{|s'|}
    + \landau(\eps^2)
    =
    s +  \frac{1}{\kappa} N + \landau(\eps^2).
  $$
  Consequently, the distance between the center $m_0$ of the discrete
  curvature circle $k_0$ and the center $s + \frac{1}{\kappa} N$ of the
  smooth curvature circle is of magnitude $\landau(\eps^2)$.

  Let us now compute the radius of the discrete curvature circle $k_0$ by
  computing the distance of its center $m_0$ to a point on the circle,
  e.g., $p_{bc}$:
  $$
    \frac{1}{\kappa_0}
    =
    |p_{bc} - m_0| 
    =
    \Big|\underbrace{p_{bc} - \big(s}_{\landau(\eps^3)} 
    + \frac{1}{\kappa} N \big)
    + \underbrace{\big(s + \frac{1}{\kappa} N \big) - m_0}_{\landau(\eps^2)}
    \Big| 
    =
    \Big|\frac{1}{\kappa} N + \landau(\eps^2)\Big| 
    =
    \frac{1}{|\kappa|} + \landau(\eps^2),
  $$
  which implies Equation~\eqref{eq:curvature}.
  Equation~\eqref{eq:mainpkt} follows from Lemma~\ref{lem:pbc}.
\end{proof}

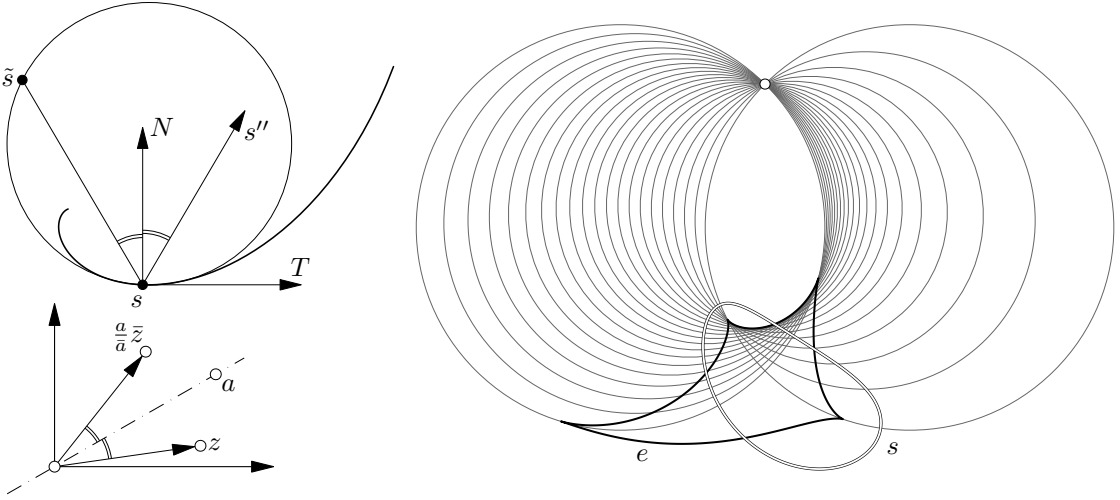
\begin{figure}[t]
  \begin{minipage}{.35\textwidth}
  \begin{overpic}[width=\textwidth]{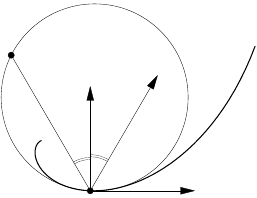}
    \put(-1,54){$\tilde s$}
    \put(32,-3){$s$}
    \put(61,40){$s''$}
    \put(37,41){$N$}
    \put(73,5){$T$}
  \end{overpic}

  \begin{overpic}[width=.7\textwidth]{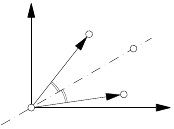}
    \put(74,18){$z$}
    \put(79,40){$a$}
    \put(39,58){$\frac{a}{\bar a} \bar z$}
  \end{overpic}
  \end{minipage}
  \hfill
  \begin{minipage}{.64\textwidth}
  \begin{overpic}[width=\textwidth]{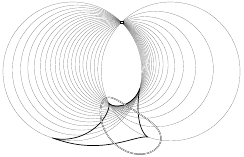}
    \put(67,3){$s$}
    \put(32,2){$e$}
  \end{overpic}
  \end{minipage}
  \caption{\emph{Top-left:} Smooth curve $s$ with curvature circle. The point
  $\tilde s = s - \frac{2 s'^2}{s''}$ on the curvature circle is
  M{\"o}bius invariantly connected to the parametrization of the curve.
  The vectors $\tilde s - s$ and $s''$ are symmetric, up to length, with
  respect to the normal vector $N$.
  \emph{Bottom-left:} Reflection. A complex point $z$ gets reflected
  to $\frac{a}{\bar a} \bar z$ along an axis through the origin and $a$.
  \emph{Right:} Illustration of a curve $s$ together with a family of
  circles which are orthogonal to $s$ and which pass through $\tilde
  s$.  All these circles pass through a fixed point which implies that
  the parametrization $s$ is M{\"o}bius equivalent to an arc length
  parametrization $\hat s$.
  The circles envelope the M{\"o}bius transformation of the evolute $e$ of
  $\hat s$.
  }
  \label{fig:moebiuspkt}
\end{figure}

In Lemma~\ref{lem:pda} we saw that the point $p_{da} = p_{\gamma_{i + 2}
\gamma_{i - 1}}$ is a second order approximation of $s(u) - \frac{2
s'^2(u)}{s''(u)}$.  In the following proposition we study the geometric
meaning of that special point.

\begin{prop}
  \label{prop:krkrpkt}
  Let $s: \R \to \C$ be a smooth curve. Then for all $u \in \R$ the point
  $$
  \tilde s(u) := s(u) - \frac{2 s'^2(u)}{s''(u)}
  $$ 
  is a point on the curvature circle at $s(u)$ (see
  Figure~\ref{fig:moebiuspkt} top\dash left).
  The curve $\tilde s$ is M{\"o}bius-invariantly connected to the
  \emph{parametrization} of $s$. 
  Furthermore, the normal vector $N$ of $s$ is the angle bisector of
  $\tilde s - s$ and the second derivative vector $s''$ (see
  Figure~\ref{fig:moebiuspkt} top\dash left).
\end{prop}
\begin{proof}
  To show that $\tilde s$ lies on the curvature circle we show
  $\big|\tilde s - (s + \frac{1}{\kappa} N) \big| = \frac{1}{|\kappa|}$:
  \begin{align*}
    \Big|\tilde s - \Big(s + \frac{1}{\kappa} N\Big) \Big| 
    =&
    \Big|\frac{2 s'^2}{s''} + \frac{1}{\kappa} N\Big| 
    =
    \Big|\frac{2 s'^2}{s''} + \frac{|s'|^3}{\frac{i}{2} (s' \bar s'' -
    \bar s' s'')} i \frac{s'}{|s'|} \Big| 
    \\
    =&
    \Big|\frac{2 s'^3 \bar s'' - 2 s'^2 \bar s' s'' + 2 s'^2 \bar s' s''}
    {s'' (s' \bar s'' - \bar s' s'')} \Big| 
    =
    \frac{|\bar s''|}{|s''|}
    \frac{|s'|^3}{|\frac{i}{2} (s' \bar s'' - \bar s' s'')|}
    =
    \frac{1}{|\kappa|}.
  \end{align*}
  Next we show the M{\"o}bius invariant property of $\tilde s$. For that,
  let $M$ be a M{\"o}bius transformation. We have to show 
  $$
  M \circ \tilde s = M \circ s - \frac{2 (M \circ s)'^2}{(M \circ s)''}.
  $$
  This equation holds trivially for translations, rotations and scalings.
  So the only thing left to show is that it is also true for inversions $M(z)
  = 1/z$. We start with the right hand side:
  $$
    \frac{1}{s} 
    - 
    \frac{2 (\frac{1}{s})'^2}{(\frac{1}{s})''}
    =
    \frac{1}{s} 
    + 
    \frac{2 \frac{s'^2}{s^4}}{\frac{s'' s^2 - 2 s s'^2}{s^4}}
    =
    \frac{s''}{s'' s - 2 s'^2}
    =
    \frac{1}{s - \frac{2 s'^2}{s''}}
    =
    \frac{1}{\tilde s}
    =
    M \circ \tilde s.
  $$

  Now we show the symmetry property.  We have to show that $\frac{\tilde s
  - s}{|\tilde s - s|}$ gets reflected to $\frac{s''}{|s''|}$ at the
  symmetry axis $N$.  The reflection of a complex number $z$ on an axis
  with direction $a$ (see Figure~\ref{fig:moebiuspkt} bottom\dash left) is
  expressible in complex numbers by $\frac{a}{\bar a} \bar z$. So we have
  to show
  $$
    \frac{N}{\overline{N}} 
    \overline{\frac{\tilde s - s}{|\tilde s - s|}}
    =
    \frac{s''}{|s''|}.
  $$
  This equation is equivalent to
  $$
    \frac{\phantom{-}i s'}{-i \bar s'}
    \frac{-\frac{2 s'^2}{s''}}{|\frac{2 s'^2}{s''}|}
    =
    \frac{s''}{|s''|}
    \quad\Leftrightarrow\quad
    \frac{s'}{\bar s'}
    \frac{\bar s'^2 |s''|}{\bar s'' |s'^2|}
    =
    \frac{s''}{|s''|}
    \quad\Leftrightarrow\quad
    \frac{s'}{\bar s'}
    \frac{\bar s'^2}{\bar s'' s' \bar s'}
    =
    \frac{s''}{s'' \bar s''},
  $$
  which is true and therefore implies the symmetry property.
\end{proof}

\begin{rem}
  If $s$ is parametrized proportionally to arc length, then $s \tilde s$
  is a diameter of the curvature circle.
\end{rem}

\begin{cor}
  \label{cor:arclength}
  A parametrized curve is M{\"o}bius equivalent to a arc length
  parametrized curve if and only if for all $u \in \R$ the circles
  orthogonal to the curvature circle and passing through $s$ and $\tilde
  s$ intersect in one common point (see Figure~\ref{fig:moebiuspkt} right).
\end{cor}

The commonly used characterization of arclength parametrizations of discrete
curves is by a polygon with constant edgelengths. However,
Corollary~\ref{cor:arclength} implies an immediate alternative:
\begin{defn}
  We call a discrete curve \emph{parametrized proportionally to arclength}
  if $p_{\gamma_i \gamma_{i + 1}}$ and $p_{\gamma_{i + 2} \gamma_{i - 1}}$
  are opposite points on the discrete curvature circle.
\end{defn}

\begin{thm}
  \label{thm:zweibein}
  The unit tangent vector $T_i$ and unit normal vector $N_i$ of the
  discrete curvature circle $k_i$ at $p_{\gamma_i \gamma_{i + 1}}$ are 
  second order approximations of the unit tangent vector $T$ and unit
  normal vector $N$ of the smooth curve (after appropriate orientation),
  i.e.,
  $$
    T_i = T + \landau(\eps^2)
    \quad\text{and}\quad
    N_i = N + \landau(\eps^2).
  $$
\end{thm}
\begin{proof}
  The approximation quality of $T$ and $N$ is the same so we just have to
  prove it for one of them:
  $$
    N_i
    = 
    \frac{p_{\gamma_i \gamma_{i + 1}} - m_i}{|p_{\gamma_i \gamma_{i + 1}}
    - m_i|}
    =
    \frac{s + \landau(\eps^3) - \big(s - \frac{1}{\kappa} N +
    \landau(\eps^2)\big)}
    {|s + \landau(\eps^3) - \big(s - \frac{1}{\kappa} N +
    \landau(\eps^2)\big)|}
    = 
    \frac{|\kappa|}{\kappa} N + \landau(\eps^2).
  $$
  After appropriate orientation $N$ and $N_i$ differ only about
  $\landau(\eps^2)$.
\end{proof}

\section{Curvature for Three Dimensional Curves}
\label{sec:3dcurves}

Before we generalize discrete curvature from discrete planar curves to
space curves we need some more results on the quaternionic cross\dash
ratio for points in three dimensional space. We will use the imaginary
quaternions $\IM \HH$ to describe points in three dimensional space $\R^3$
(see Sec.~\ref{subsec:quaternions}).

\subsection{Cross-ratio and geometry}

The authors used the quaternionic algebra and the cross\dash ratio
extensively in~\cite{vaxman+2017,vaxman+2018} for applications in regular
mesh design and for M{\"o}bius invariant subdivision algorithms. The
results of this paragraph can also be found there.  To prove some
technical Lemmas we first consider the following geometric property, which
can easily be verified.
\begin{lem}
	\label{lem:elemgeom}
	Let $a, b, c \in \R^n$ be three points. Then 
	$
	(a - b) \|a - c\|^2 - (a - c) \|a - b\|^2
	$
	is the direction of the tangent of the circumcircle  to the triangle $a
  b c$ at $a$.
\end{lem}

\begin{defn}
  Let $a, b, c \in \IM \HH$ be pairwise distinct points. Then we call the
  imaginary quaternion
  $$
    t[a, b, c] := (a - b)^{-1} + (b - c)^{-1},
  $$
  \emph{corner tangent}.
\end{defn}

Note that the identity $a^{-1} + b^{-1} = a^{-1} (a + b) b^{-1}$
immediately implies 
\begin{equation}
  \label{eq:cornertangent}
  t[a, b, c] = (a - b)^{-1} (a - c) (b - c)^{-1}.
\end{equation}

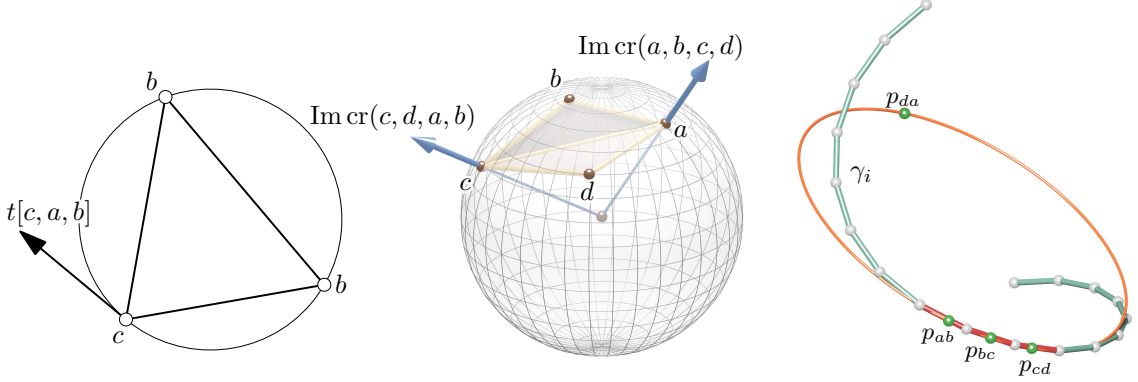
\begin{figure}[t]
  \begin{overpic}[width=.3\textwidth]{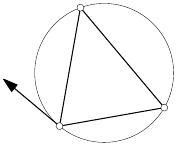}
    \lput(33,4){$c$}
    \rput(96,19){$b$}
    \lput(43,80){$b$}
    \rput(-2,41){\contour{white}{$t[c, a, b]$}}
  \end{overpic}
  \hfill
  \begin{overpic}[width=.3\linewidth]{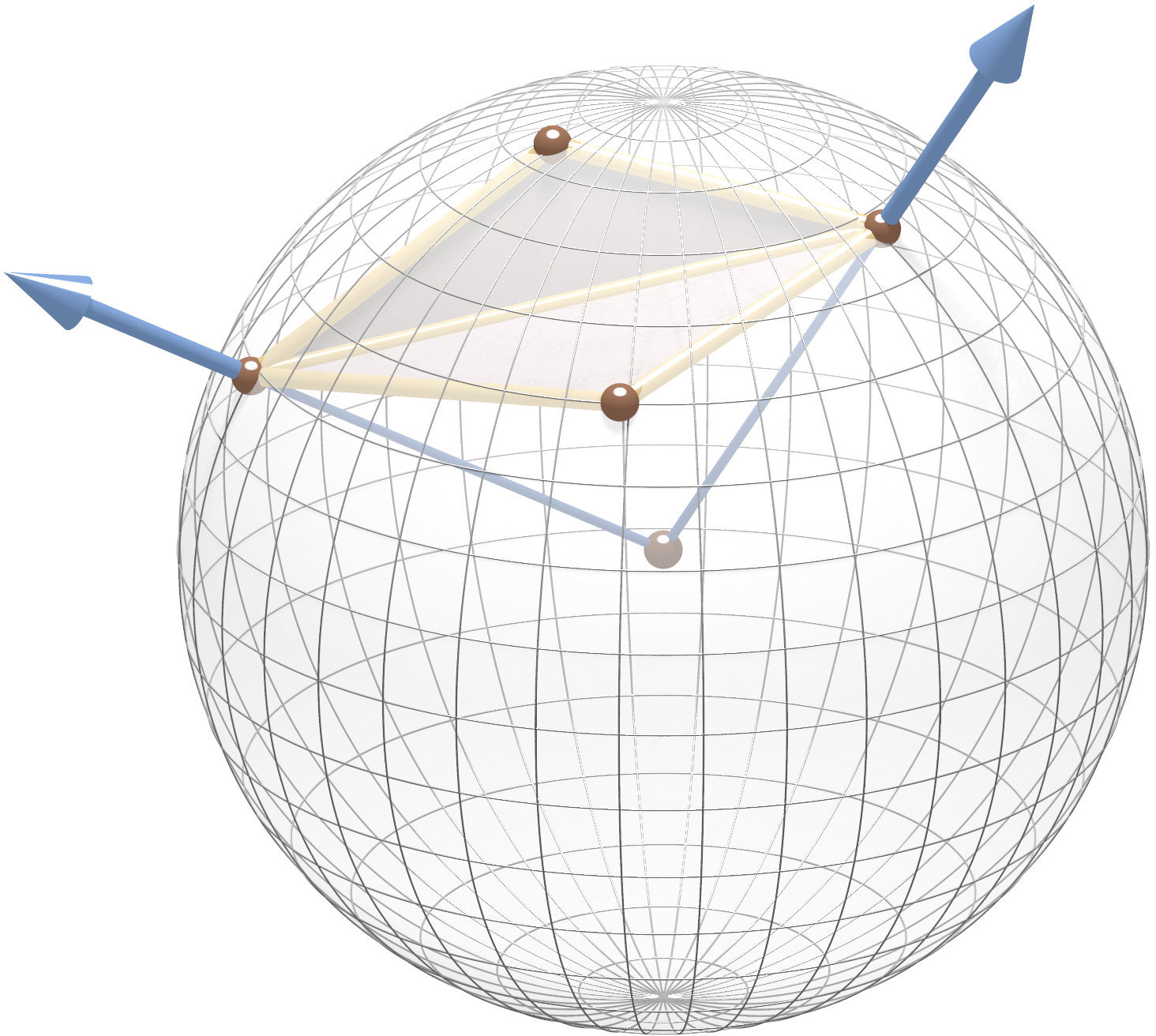}
    \put(79,66){\contour{white}{$a$}}
    \put(42,81){\contour{white}{$b$}}
    \put(15,50){\contour{white}{$c$}}
    \put(51,47){\contour{white}{$d$}}
    \lput(100,91){\contour{white}{$\IM\crr(a, b, c, d)$}}
    \lput(20,70){\contour{white}{$\IM\crr(c, d, a, b)$}}
  \end{overpic}%
  \hfill
  \begin{overpic}[width=.3\linewidth]{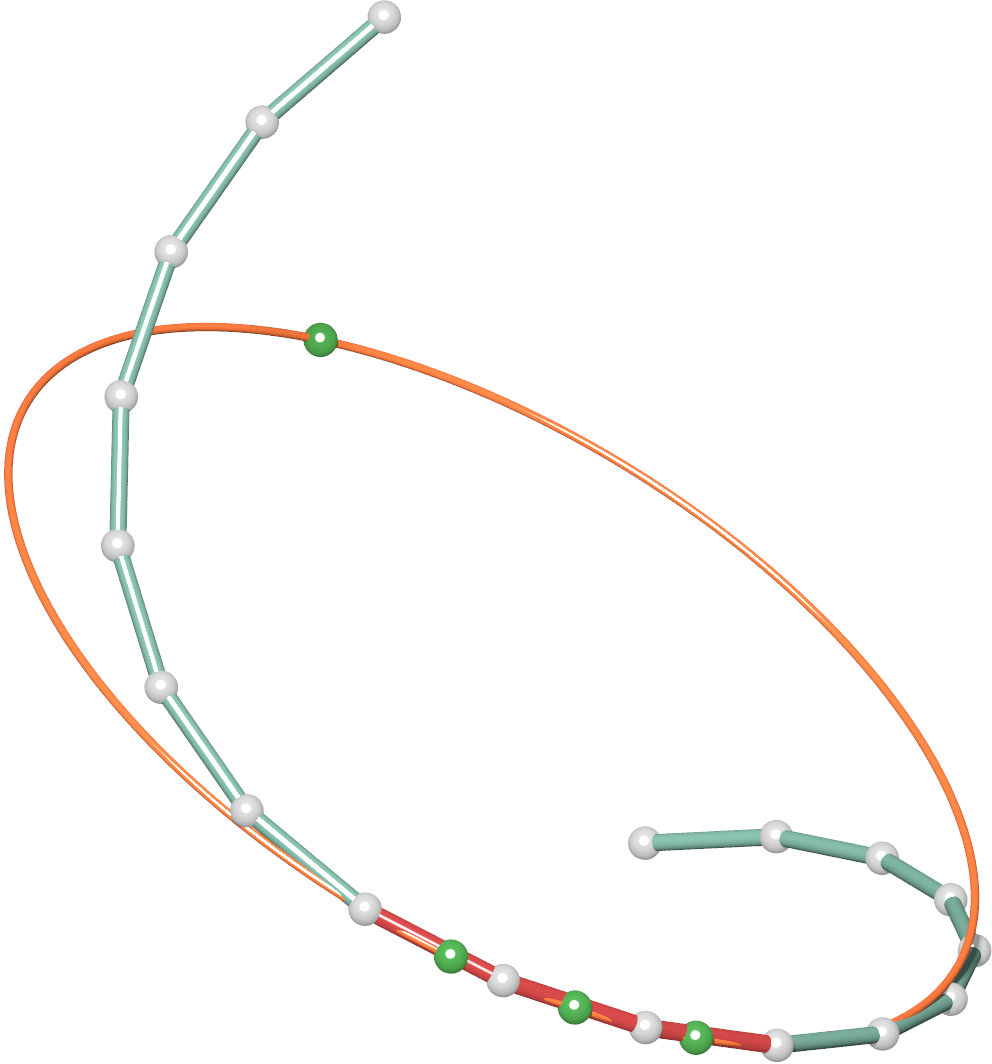}
    \put(15,50){$\gamma_i$}
    \put(35,5){\small$p_{ab}$}
    \put(47,0){\small$p_{bc}$}
    \put(62,-4){\small$p_{cd}$}
    \put(25,71){\small$p_{da}$}
  \end{overpic}%
  \caption{\emph{Left}: The corner tangent $t[c, a, b]$ is a vector
  in tangential contact with the circumcircle of a triangle $(a b c)$ at $a$.
  \emph{Center}: Circumsphere of $a, b, c, d$. The imaginary part of the
  cross\dash ratio $\crr(a, b, c, d)$ is a vector that is orthogonal to
  the circumsphere at $a$.
  \emph{Right}: A discrete space curve with a curvature circle. The four
  points $p_{ab}, p_{bc}, p_{cd}, p_{da}$ are concyclic also in the
  $3$-space case.
  }
  \label{fig:cornertangent}
\end{figure}

\begin{lem}
  \label{lem:tangentVector}
  Consider the circumcircle of $a, b, c \in \HH$, oriented according to this
  defining triangle.  Then the vector $t[c, a, b]$, placed at $a$, is in
  oriented tangential contact with the circle (see
  Figure~\ref{fig:cornertangent} left).
\end{lem}
\begin{proof}
	Note that $q\! \in\! \IM \HH$ implies $q^{-1}\!\! =\! -q/|q|^2$. 
	Using the definition of the corner tangent yields
	\begin{equation*}
    t[c, a, b] = (c - a)^{-1} + (a - b)^{-1} 
    = -(c - a)/|c - a|^2 - (a - b)/|a - b|^2.
	\end{equation*}
	Consequently, Lemma \ref{lem:elemgeom} concludes the proof.
\end{proof}

\begin{lem}
        \label{lem:normalscr}
        Let $a, b, c, d \in \IM \HH$ be four points not lying on a common
        circle.  Then, the imaginary part of the cross\dash
        ratio is the normal of the circumsphere (or plane) at $a$, i.e.,  
        for a \emph{proper} circumsphere with center $m$ we have
		$
        \IM \crr(a, b, c, d) \parallel (m - a)
    $ 
          (see Figure~\ref{fig:cornertangent} center).
\end{lem}
\begin{proof}
	We compute the cross\dash ratio in terms of corner tangents  
	abbreviated by $t_1 = t[c, a, b]$ and $t_2 = t[d, a, c]$:
	\begin{align*}
		\crr(a, b, c, d) 
    &=
    (a - b)^{\vphantom{-1}}\, (b - c)^{-1}\,
    (c - d)^{\vphantom{-1}}\, (d - a)^{-1}
    \\
    &=
    (a - b)^{\vphantom{-1}}\, (b - c)^{-1} 
    [(a - c)^{\vphantom{-1}}\, (a - c)^{-1}]
    (c - d)^{\vphantom{-1}}\, (d - a)^{-1}\\
		&=
    [(a - c)^{-1}\, (b - c)^{\vphantom{-1}}\, (a - b)^{-1}]^{-1}
    [(a - c)^{-1}\, (c - d)^{\vphantom{-1}}\, (d - a)^{-1}]\\
		&=
    [(a - c)^{-1}\, (b - c)^{\vphantom{-1}}\, (b - a)^{-1}]^{-1}
    [(a - c)^{-1}\, (d - c)^{\vphantom{-1}}\, (d - a)^{-1}]
    \\
    &\overset{\eqref{eq:cornertangent}}{=}
		t_1^{-1}\, t_2^{\vphantom{-1}}.
	\end{align*}
  Since $t_1$ and $t_2$ are both imaginary we can write the cross\dash
  ratio as
	$$
		\crr(a, b, c, d) 
    =
	  [
		  \la t_1^{-1}, t_2^{\vphantom{-1}}\ra,
		  -t_1^{-1} \times t_2^{\vphantom{-1}}
	  ].
	$$
	Lemma~\ref{lem:tangentVector} implies that $t_1^{-1}$ and
	$t_2$ are tangent vectors to the circumcircles of the triangles 
	$(a b c)$ and $(c d a)$, respectively, both at $a$.
  Consequently, the imaginary part of the above cross\dash ratio is the
  cross product of tangent vectors to circles on the circumsphere of $a,
  b, c, d$ at $a$, hence orthogonal to the tangent plane of the
  circumsphere at $a$.
\end{proof}

\begin{prop}
  \label{prop:spherical}
  Let $a, b, c, d \in\IM\HH$ be four non\dash concyclic points with
  $\crr(a, b, c, d) = [r, v]$. Further, let $f \in \HH$ be the quaternion
  that solves 
	\begin{equation*}
		\crr(a, b, c, f) = [\lambda r, \mu v],
	\end{equation*}
  for some $\lambda, \mu \in \R$.
	Then $f \in \IM \HH$, i.e., $f$ is an imaginary quaternion representing
  a point in $\R^3$. Furthermore, $f$ lies on the circumsphere of $a, b,
  c, d$. In particular $f(\lambda, \mu)$ is a parametrization of the
  circumsphere.
\end{prop}
\begin{proof}
  The two occurring cross\dash ratios can be expressed as (see the proof of
  Lemma~\ref{lem:normalscr})
	\begin{align*}
		\crr(a, b, c, d) = t_1 \cdot t_2,
		\quad\text{and}\quad
		\crr(a, b, c, f)= t_1 \cdot t_3,
	\end{align*}
  where 
  $t_1 := t[c, a, b]^{-1}$,
  $t_2 := t[d, a, c]$, and 
  $t_3 := t[f, a, c]$. 
  Consequently, as all $t_i \in \IM \HH$, we have 
	\begin{equation*}
		[r, v] = [\la t_1, t_2\ra, -t_1 \times t_2],
		\quad\text{and}\quad
		[\lambda r, \mu v] 
		= [\la t_1, t_3\ra, -t_1 \times t_3].
	\end{equation*}
  Since all $t_1, t_2, t_3$ are orthogonal to $v$ and are therefore 
  linearly dependent we can express $t_3$ in the form $t_3 = \alpha t_1 +
  \beta t_2$.
  The two vectors $t_1$ and $t_2$ are linearly independent since otherwise
  $t_1 \times t_2$ would be zero and therefore $\crr(a, b, c, d) = [r, v]
  = [-\la t_1, t_2\ra, 0] \in \R$ which is a contradiction to the four
  points $a, b, c, d$ not being concyclic.

  After inserting $t_3 = \alpha t_1 + \beta t_2$ into the above equations
  we obtain
	\begin{align*}
		&
		\lambda \la t_1, t_2\ra 
		= 
    \lambda r
		= 
    \la t_1, t_3\ra
		= 
		\alpha \la t_1, t_1\ra + \beta \la t_1, t_2\ra,
		\\
		&
		\mu t_1 \times t_2 
		=
    -\mu v
		=
    t_1 \times t_3
		= 
		\alpha t_1 \times t_1 + \beta t_1 \times t_2.
	\end{align*}
	Consequently, $\beta = \mu$ and 
	$\alpha = (\lambda - \mu) \la t_1, t_2\ra/|t_1|^2$, which determines
	$t_3$ uniquely. From the definition of  
  $t_3 = t[f, a, c] = (f - a)^{-1} + (a - c)^{-1}$, we then
	immediately get
	\begin{equation*}
		f = (t_3 - (a - c)^{-1})^{-1} + a \in \IM \HH.
	\end{equation*}
	Furthermore, the circumsphere of $a, b, c, f$ is the same as the
	circumsphere of $a, b, c, d$ since both pass through $a, b, c$ and both
	have parallel normal vectors ($\mu v$ and $v$, resp.) at $a$, and there is
	only one such sphere.
\end{proof}

\subsection{Point-insertion-rule in $\HH$}

Let us now consider the analogous construction of~\eqref{eq:inserting} by 
inserting a new point to given four points $a, b, c, d \in \IM \HH$ in
three dimensional space. However, the quaternionic square root is not
uniquely defined in our formulation (see Sec.~\ref{subsec:quaternions})
for negative real numbers. So we must exclude that case in the following
which is not a significant restriction as this case (i.e., $\crr(c, a, b,
d) \in \R_{< 0})$ corresponds to a concyclic quadrilateral $a, b, c, d$
with $a$ separated from $d$ by $b$ and $c$ on the circumcircle. We exclude
such ``zigzag'' quadrilaterals in the following and consider them as
discrete singularities of our polygons.

The quaternionic formula analogous to~\eqref{eq:inserting} reads:
\begin{equation*}
  f(a, b, c, d) := 
  \big(
    (b - a) (c - a)^{-1} \sqrt{\crr(c, a, b, d)} + 1
  \big)^{-1}
  \cdot
  \big(
    (b - a) (c - a)^{-1} \sqrt{\crr(c, a, b, d)} c + b
  \big).
\end{equation*}
The notation of this formula is less flexible than in the complex case due
to the noncommutativity of $\HH$.
As it will turn out $f(a, b, c, d)$ is purely imaginary and thus in three
space, but note that a priory $f$ is a quaternion and at first not
apparently imaginary. 
In analogy to Lemma~\ref{lem:crequation} $f$ is also the solution to a
cross\dash ratio equation:
\begin{lem}
  \label{lem:spherical}
  The newly inserted point $f(a, b, c, d)$ fulfills
  $$
    \crr(c, a, b, f(a, b, c, d)) = -\sqrt{\crr(c, a, b, d)}.
  $$
\end{lem}

\begin{cor}
$f$ is a point in three dimensional space, i.e., $f \in \IM \HH$. Even
  more, $f$ lies on the circumsphere of $a, b, c, d$.
\end{cor}
\begin{proof}
  The square root of a quaternion $q = [r, v]$ (see
  Sec.~\ref{subsec:quaternions}) is a quaternion with imaginary part
  parallel to $v$, i.e., parallel to the imaginary part of $q$.
  Consequently, Proposition~\ref{prop:spherical} implies that $f$ is in
  $\IM \HH$ and in particular on the circumsphere of $a, b, c, d$.
\end{proof}

\subsection{Curvature for discrete space curves}

In this section we will relate the curvature and curvature circle of
discrete curves in three dimensional space to the planar case
(Sec.~\ref{subsec:curvature}). 
But first let us recall some properties of smooth curves $s: \R \to \R^3$.

Consider a sequence of four points on the curve $s$ which converge to one
point $s(0)$. At any time the four points are assumed to uniquely
determine a sphere.  Consequently, as the four points converge to one
point the sequence of spheres defined that way converges to the so called
osculating sphere (see e.g.,~\cite{docarmo-1976}). The \emph{osculating
sphere} passes through $s(0)$ and has its center at
\begin{equation}
  \label{eq:osculatingsphere}
  s(0) + \frac{1}{\kappa} N + \frac{\kappa'}{\kappa^2 \tau} B,
\end{equation}
where $N$ is the unit normal vector, $B$ the binormal unit vector,
$\kappa$ the curvature, and $\tau$ the torsion of the curve. The
curvature circle at $s(0)$ is the intersection of the osculating plane
with the osculating sphere and thus lies on the osculating sphere.

\begin{lem}
  \label{lem:osculating}
  The osculating sphere has contact of order $\geq 3$ with the curve $s$
  which implies that there is a curve $\hat s$ on the osculating sphere,
  s.t.,
  $$
  s(0) = \hat s(0),
  \quad
  s'(0) = \hat s'(0),
  \quad
  s''(0) = \hat s''(0),
  \quad
  s'''(0) = \hat s'''(0),
  $$
\end{lem}

This immediately implies the following lemma.
\begin{lem}
  The curvature and the curvature circle of a space curve $s(u)$ at $u =
  0$ is the same as the curvature and the curvature circle of $\hat s$ on the
  osculating sphere at $u = 0$.
\end{lem}

Any M{\"o}bius transformation that maps the osculating sphere to a plane
also transforms the curvature circle to that plane.

Let us now define a curvature circle for discrete space curves. 
So let us start with a discrete curve $\gamma: \Z \to \R^3$ and set 
$
a = \gamma_{i - 1}, 
b = \gamma_{i}, 
c = \gamma_{i + 1}, 
d = \gamma_{i + 2}.
$
In analogy to Theorem~\ref{thm:circle} we define 
$$
  p_{ab} = f(d, a, b, c),\
  p_{bc} = f(a, b, c, d),\
  p_{cd} = f(b, c, d, a),\
  p_{da} = f(c, d, a, b),
$$
but now for the `quaternionic' $f$. Lemma~\ref{lem:spherical} implies that
$p_{ab}, p_{bc}, p_{cd}, p_{da}$ lie on the circumsphere of $a, b, c, d$
which we consider as the \emph{discrete osculating sphere}. 

Let us now consider a M{\"o}bius transformation that maps the osculating
sphere to the $[yz]$-plane of a Cartesian $xyz$-coordinate system. This
M{\"o}bius
transformation (as any M{\"o}bius transformation does) keeps the real part
as well as the length of the imaginary part of the cross\dash ratio of
four points invariant. The transformed cross\dash ratios have imaginary parts
that are orthogonal to the circumsphere of the new points
(Lemma~\ref{lem:normalscr}). Therefore the transformed cross\dash ratios
have imaginary parts that are parallel to the $x$-axis of the coordinate
system. Consequently, the cross\dash ratios are complex numbers $[r, (x,
0,0)]$ and we arrive at the case of planar curves
(Sec.~\ref{sec:curvatureplanar}).

So after the M{\"o}bius transformation we can apply
Theorem~\ref{thm:circle} which implies that $p_{ab}, p_{bc}, p_{cd},
p_{da}$ lie on a common circle $\tilde k_i$ and have a cross\dash ratio of
$-1$.  Furthermore, the inverse M{\"o}bius transformation maps the circle
$\tilde k_i$ to a circle $k_i$ on the osculating sphere. And since
M{\"o}bius transformations map the curvature circle of a curve to the
curvature circle of the transformed curve the following definition is
sensible.

\begin{defn}
  For discrete space curves $\gamma: \Z \to \R^3$ we call the circle $k_i$
  \emph{(discrete) curvature circle} and the inverse of its radius
  \emph{curvature $\kappa_i$ at the edge $\gamma_{i}\gamma_{i + 1}$}. For
  an illustration see Figure~\ref{fig:cornertangent} (right).
\end{defn}

\begin{thm}
  \label{thm:mainspacial}
  Let $s: \R \to \R^3$ be a sufficiently smooth planar curve and let $u,
  \eps \in \R$. Let further $\gamma: \Z \to \R^3$ be the discrete curve
  $\gamma_k = \gamma(k) = s(u + (2 k - 1)\eps)$.
  All the approximation results from Theorem~\ref{thm:main} apply to space
  curves in $\R^3$.
\end{thm}
\begin{proof}
  At first we convince ourselves that it is sufficient to replace the
  curve $s$ by the curve $\hat s$ on the osculating sphere
  (Lemma~\ref{lem:osculating}).
  So instead of $\gamma_k$ we use $\hat \gamma_k = \hat s(u + (2 k -
  1)\eps)$ for the computation of the discrete curvature circle.
  We have $\hat \gamma_k = \gamma_k + \landau(\eps^4)$ and therefore
  $$
    f(\hat \gamma_{i - 1}, \hat \gamma_i, \hat \gamma_{i + 1}, 
    \hat \gamma_{i + 2})
    =
    f(\gamma_{i - 1}, \gamma_i, \gamma_{i + 1}, \gamma_{i + 2})
    + \landau(\eps^4),
  $$
  i.e., the four points for which we construct the discrete curvature circle
  are $\landau(\eps^4)$-close to the points on the actual discrete
  curvature circle. And the center of the replacing curvature circle is
  therefore also $\landau(\eps^4)$-close to the actual circle since the
  center of the circumcircle of a triangle $a, b, c \in \R^3$ is
  $$
    \frac{(\|a - c\|^2 (b - c) - \|b - c\|^2 (a - c)) \times ((a - c)
    \times (b - c))}{2 \|(a - c) \times (b - c)\|^2} + c.
  $$
  So now we know that it is sufficient to show the $3$-space version of
  Theorem~\ref{thm:main} for $\hat s$ instead of $s$. After a
  stereographic projection from the osculating sphere to the complex
  plane we arrive at the case of planar curves
  (Sec.~\ref{sec:curvatureplanar}) for which Theorem~\ref{thm:main} holds.
  The only thing left to prove is that the stereographic projection does
  not change the approximation order of the center of the curvature
  circle. 
  
  Let $m, r$ denote the center and radius of the smooth curvature circle
  of the planar curve, and let $m_0(\eps), r_0(\eps)$ denote the curvature
  circle of $\gamma$. From Theorem~\ref{thm:main} we know that 
  $$
    m = m_0(\eps) + \landau(\eps^2)
    \quad\text{and}\quad
    r = r_0(\eps^2) + \landau(\eps^2).
  $$
  After mapping a circle in $\C$ with center $m = m_1 + i m_2$ and radius
  $r$ stereographically to the sphere we obtain 
  $$
    \phi(m, r) 
    :=
    \frac{1}{r^2 - 2 r^2 (|m|^2 - 1) + (|m|^2 + 1)^2}
    \left(
    \begin{array}{c}
      2 m_1 (1 - r^2 + |m|^2)
      \\
      2 m_2 (1 - r^2 + |m|^2)
      \\
      (r^2 - 1 - |m|^2) (r^2 + 1 - |m|^2)
    \end{array}
    \right)
  $$
  for the new center. Therefore 
  $$
    \phi(m_0, r_0) = \phi(m, r) + \landau(\eps^2),
  $$
  i.e., the centers are $\landau(\eps^2)$-close.
\end{proof}

\section{Torsion}
\label{sec:torsion}

We define the torsion for discrete curves again partially in terms of the
cross\dash ratio. But before we do that we have to consider the right
formulation of the torsion of smooth curves.

\subsection{Torsion for smooth curves}

Let us reformulate the common notation of the torsion $\tau$ (see
Equation~\eqref{eq:curvaturetorsion}):
\begin{align*}
  \tau 
  = 
  -\frac{\la s' \times s'', s'''\ra}{\|s' \times s''\|^2}
  = 
  -\frac{\det(s', s'', s''')}{\|s' \times s''\|^2}
  = 
  -\frac{\det(s''', s', s'')}{\|s' \times s''\|^2}
  = 
  \frac{\la s' \times s''', s''\ra}{\|s' \times s''\|^2}
  = 
  \frac{\la s' \times s''', s''\ra}{\kappa^2 \|s'\|^6}.
\end{align*}
The normal unit vector $N$ is the cross product of the binormal unit
vector $B$ and the tangent unit vector $T$ and therefore reads
\begin{align*}
  N 
  = 
  B \times T
  =
  \frac{s' \times s''}{\|s' \times s''\|}
  \times
  \frac{s'}{\|s'\|}
  =
  \frac{s'' \la s' ,s' \ra - s' \la s', s''\ra}{\|s'\| \|s' \times s''\|}
  =
  \frac{1}{\|s'\|^2 \kappa} s''
  -\frac{\la s', s'' \ra}{\|s'\| \|s' \times s''\|} s'.
\end{align*}
Consequently,
\begin{equation}
  \label{eq:tau}
  \tau 
  = 
  \frac{\la s' \times s''', s''\ra}{\kappa^2 \|s'\|^6}
  = 
  \frac{\la s' \times s''', \|s'\|^2 \kappa N\ra}{\kappa^2 \|s'\|^6}
  = 
  \frac{\la s' \times s''', N\ra}{\kappa \|s'\|^4}.
\end{equation}
We will come back to that formulation of $\tau$ in the proof of
Theorem~\ref{thm:torsion}.

\subsection{Discrete Frenet frame}
\label{subsec:frenetframe}

There is a natural way in our setting to define a discrete Frenet frame.
Theorem~\ref{thm:main} implies that $p_{bc} = p_{\gamma_{i} \gamma_{i +
1}}$ is a good discrete candidate for a point where the curvature circle
should be in tangential contact with the curve as $p_{bc}$ is a third
order approximation of $s(u)$. So it is sensible to choose the discrete
unit tangent vector $T_i$ to be in tangential contact with the curvature
circle at $p_{bc}$. It is therefore equally natural to define the normal
unit vector $N_i$ to be the normal of the curvature circle at $p_{bc}$.
Consequently, the binormal vector $B_i$ should be orthogonal to $N_i$ and
$T_i$ (see Figure~\ref{fig:frenet}).

\begin{figure}[t]
  \begin{minipage}[t]{.25\textwidth}
    \begin{overpic}[width=\textwidth]{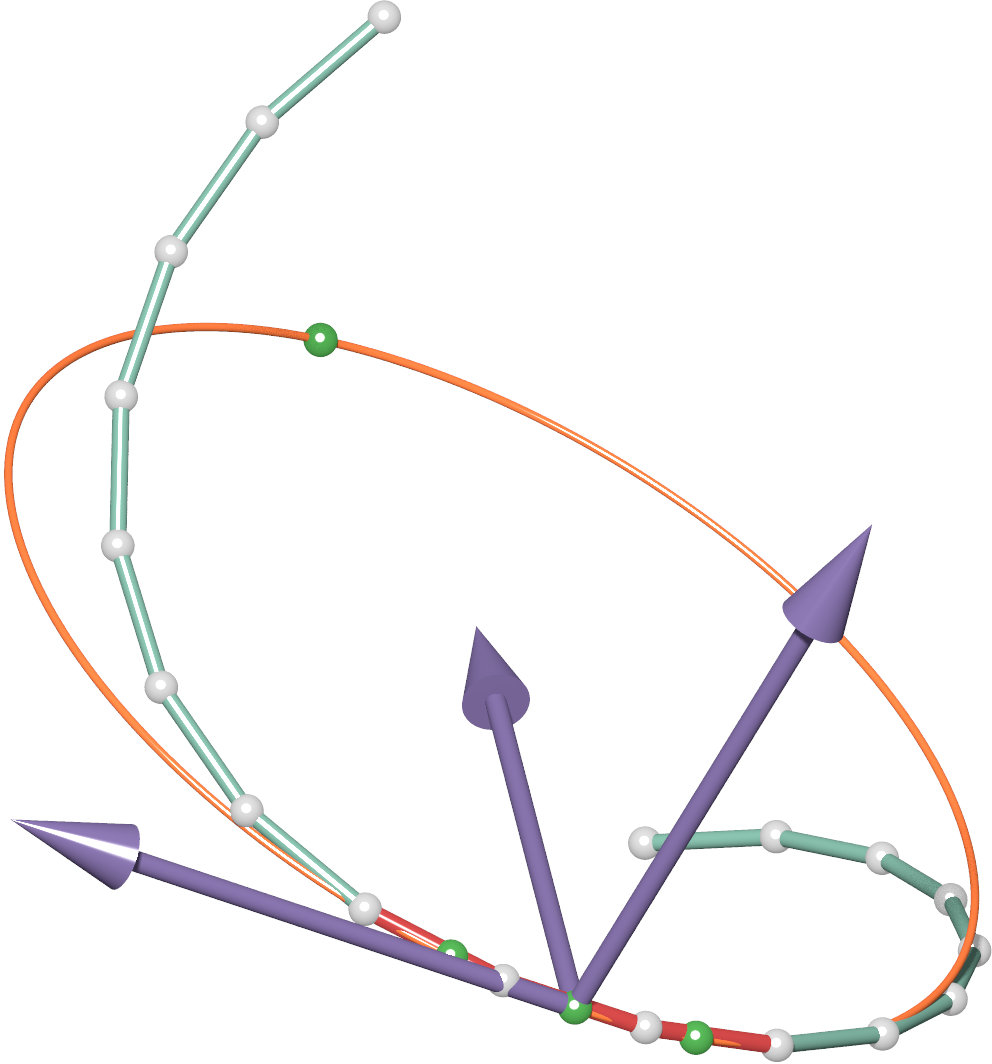}
      \put(4,7){\small$T$}
      \put(83,48){\small$B$}
      \put(50,37){\small$N$}
      \put(39,15){\small$p_{ab}$}
      \put(47,-1){\small$p_{bc}$}
      \put(64,6){\small$p_{cd}$}
      \put(26,60){\small$p_{da}$}
      \put(63,55){\small$k_i$}
    \end{overpic}
  \end{minipage}
  \hfill
  \begin{minipage}[b]{.71\textwidth}
  \caption{A discrete space curve with discrete curvature
  circle $k_i$. The discrete tangent vector $T$ of the curve is defined to
  be the tangent vector of the curvature circle at $p_{bc}$. The discrete
  normal vector $N$ lies in the plane of the circle and orthogonal to $T$
  and the binormal vector $B$ is orthogonal to both.
    }
  \label{fig:frenet}
  \end{minipage}
\end{figure}
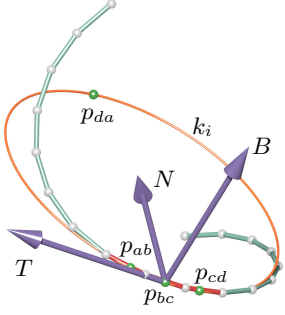

\begin{lem}
  \label{lem:n}
  Let $s: \R \to \R^3$ be a sufficiently smooth curve and let $u,
  \eps \in \R$. Let further $\gamma: \Z \to \C$ be the discrete curve
  $\gamma_k = \gamma(k) = s(u + (2 k - 1)\eps)$.
  Then the discrete unit normal $N_i$ is a second order approximation of the
  smooth normal $N$, i.e.,
  $$
  N_i = N + \landau(\eps^2).
  $$
\end{lem}
\begin{proof}
  $$
    N_i 
    =
    \frac{p_{bc} - m_0}{\|p_{bc} - m_0\|}
    \overset{(*)}{=}
    \frac{s + \landau(\eps^3) - m + \landau(\eps^2)}
    {\|s + \landau(\eps^3) - m + \landau(\eps^2)\|}
    =
    \frac{s - m}{\|s - m\|} + \landau(\eps^2)
    =
    N + \landau(\eps^2),
  $$
  where we used Theorem~\ref{thm:mainspacial} at $(*)$.
\end{proof}

\begin{lem}
  With the same assumptions as in Lemma~\ref{lem:n} we obtain
  $$
  T_i = T + \landau(\eps^2).
  $$
\end{lem}
\begin{proof}
  Theorem~\ref{thm:zweibein} implies $T_i = T + \landau(\eps^2)$ for the
  planar case. What remains to verify is that a M{\"o}bius transformation
  does not change this order.

  Any vector $v$ attached at a point $p$ can be represented as the
  derivative of a straight line:
  $$
    [p + t v]_{t = 0}.
  $$
  Consequently, an inversion maps that vector to 
  $$
    \left[\frac{d}{dt} \frac{p + t v}{\|p + t v\|^2}\right]_{t = 0}
    =
    \frac{\|p\|^2 v - 2 \la p, v \ra p}{\|p\|^4}.
  $$
  In our case the vector $T_i$ is attached at point $p_{bc}$.
  Since $T_i = T + \landau(\eps^2)$ and $p_{bc} = s + \landau(\eps^3)$ for
  planar curves, we obtain for the tangent vector after inversion
  \begin{align*}
    \frac{\|p_{bc}\|^2 T_i - 2 \la p_{bc}, T_i \ra p_{bc}}{\|p_{bc}\|^4}
    &=
    \frac{\|s + \landau(\eps^3)\|^2 (T + \landau(\eps^2))
    - 2 \la s + \landau(\eps^3), T + \landau(\eps^2) \ra (s +
    \landau(\eps^3))}{\|s + \landau(\eps^3)\|^4}
    \\
    &=
    \frac{\|s\|^2 T - 2 \la s, T\ra s}{\|s\|^4} + \landau(\eps^2).
  \end{align*}
  Therefore, M{\"o}bius transformations map $\eps^2$-close vectors
  attached at $\eps^3$-close points to $\eps^2$-close vectors.
\end{proof}

\begin{cor}
  With the same assumptions as in Lemma~\ref{lem:n} the discrete Frenet
  frame $(T_i, N_i, B_i)$ is a second order approximation of the smooth
  Frenet frame $(T, N, B)$.
\end{cor}

\subsection{Torsion for discrete curves}

In this section we relate the torsion of a discrete curve to the
cross\dash ratio of four successive vertices of the curve. As the
real part and the length of the imaginary part of the quaternionic
cross\dash ratio is M{\"o}bius invariant but the torsion is not, the
definition must also include other quantities that are not M{\"o}bius
invariant, in our case curvature and length.  In Theorem~\ref{thm:torsion}
we again use asymptotic analysis to justify our definition of the discrete
torsion. 
\begin{defn}
  Let $\gamma: \Z \to \R^3 \cong \IM \HH$ be a discrete curve, let
  $\kappa_i$ be the discrete curvature at the edge $\gamma_{i} \gamma_{i +
  1}$, and let $N_i$ denote the discrete normal unit vector. Then, we call 
  $$
    \tau_i
    :=
    -
    \frac{9 \la \IM 
    \crr(\gamma_{i - 1}, \gamma_{i}, \gamma_{i + 1}, \gamma_{i + 2}), N_i
    \ra}{2 \kappa_i \|\gamma_{i}\! -\! \gamma_{i + 1}\|^2}
  $$
  \emph{(discrete) torsion of $\gamma$ at the edge $\gamma_{i}
  \gamma_{i + 1}$}.
\end{defn}

\begin{prop}
  The discrete torsion vanishes for planar discrete curves.
\end{prop}
\begin{proof}
  Planarity of the discrete curve and Lemma~\ref{lem:normalscr} imply that
  the imaginary part of the cross\dash ratio in the definition of the
  torsion is perpendicular to that plane. The normal vector $N_i$ on the
  other hand is contained in the plane. Therefore the two vectors are
  orthogonal and the discrete torsion vanishes.
\end{proof}

\begin{thm}
  \label{thm:torsion}
  Let $s: \R \to \R^3$ denote a sufficiently smooth curve, let $u, \eps
  \in \R$ and let the discrete curve $\gamma: \Z \to \C$ with $\gamma_k =
  \gamma(k) = s(u + (2 k - 1)\eps)$ be a sampling of $s$. Then
  $$
    \tau_0 = \tau + \landau(\eps^2).
  $$
\end{thm}

However, before we prove this theorem we need a preparatory lemma.

\begin{lem}
  Let $s: \R \to \C$ denote a sufficiently smooth curve, let $u,
  \eps \in \R$ and let the discrete curve $\gamma: \Z \to \C$ with
  $\gamma_k = \gamma(k) = s(u + (2 k - 1)\eps)$ be a sampling of
  $s$. 
  Let further $q_0$ denote the cross\dash ratio of four consecutive
  vertices $q_0 := \crr(\gamma_{-1}, \gamma_0, \gamma_1, \gamma_2)$. Then
  \begin{align*}
    \RE q_0 
    &=
    -\frac{1}{3} 
    - 
    \frac{-24 \la s', s''\ra^2 + 8 \|s'\|^2 \la s', s'''\ra + 12 \|s'\|^2
    \|s''\|^2}{9 \|s'\|^4} \eps^2 + \landau(\eps^3),
    \\
    \IM q_0
    &=
    -\frac{8 \|s'\|^2 s' \times s''' - 24 \la s', s'' \ra s' \times s''}
    {9 \|s'\|^4} \eps^2 + \landau(\eps^3).
  \end{align*}
\end{lem}
\begin{proof}
  We compute
  $$
    \crr(\gamma_{-1}, \gamma_0, \gamma_1, \gamma_2)
    =
    (\gamma_{-1} - \gamma_0)
    (\gamma_0 - \gamma_1)^{-1}
    (\gamma_1 - \gamma_2)
    (\gamma_2 - \gamma_{-1})^{-1}
  $$
  by first expressing each factor in terms of its Taylor expansion:
  \begin{align*}
    \gamma_{-1} - \gamma_0
    &= 
    -2 s' \eps + 4 s'' \eps^2 - \frac{13 s'''}{3} \eps^3 +
    \landau(\eps^4),
    \\
    \gamma_1 - \gamma_2
    &= 
    -2 s' \eps - 4 s'' \eps^2 - \frac{13 s'''}{3} \eps^3 +
    \landau(\eps^4),
  \end{align*}
  and now the inverted factors
  \begin{align*}
    (\gamma_0 - \gamma_1)^{-1}
    &=
    \Big(
      -2 s' \eps - \frac{s'''}{3} \eps^3 + \landau(\eps^4)
    \Big)^{-1}
    =
    \frac{2 s' \eps + \frac{s'''}{3} \eps^3 + \landau(\eps^4)}
    {4 \|s'\|^2 \eps^2 + \frac{4\la s', s'''\ra}{3} \eps^4
    + \landau(\eps^5)
    }
    \\
    &=
    \frac{s'}{2 \|s'\|^2 \eps}
    +
    \frac{\|s'\|^2 s''' - 2 \la s', s'''\ra s'}{12 \|s'\|^4} \eps
    + \landau(\eps^2),
  \end{align*}
  where the last equality holds since 
  $\frac{x_0 + x_2 \eps^2 + \landau(\eps^3)}{y_1 \eps + y_3 \eps^3 +
  \landau(\eps^4)} 
  = \frac{x_0}{y_1 \eps} + \frac{x_2 y_2 - x_0 y_3}{y_1^2} \eps
  + \landau(\eps^2)$.
  Analogously, we obtain
  \begin{align*}
    (\gamma_2 - \gamma_{-1})^{-1}
    = 
    -\frac{s'}{6 \|s'\|^2 \eps}
    +
    \frac{-\|s'\|^2 s''' + 2 \la s', s'''\ra s'}{4 \|s'\|^4} \eps
    + \landau(\eps^2).
  \end{align*}
  The above four factors are all purely imaginary quaternions.
  Multiplying these factors together in the right order yields the
  proposed real and imaginary part of the cross\dash ratio.
\end{proof}

So, let us now turn to our approximation result for the torsion:

\begin{proof}[Proof of Theorem~\ref{thm:torsion}]
  We show the formula at $u = 0$ and therefore $i = 0$:
  \begin{align*}
    \tau_0
    &=
    -\frac{9}{2} 
    \frac{\la \IM q_0, N_0\ra}{\kappa_0 \|\gamma_0 - \gamma_1\|^2}
    \overset{(*)}{=}
    -\frac{9}{2} 
    \frac{\la (-8 \|s'\|^2 s' \times s''' + 24 \la s', s''\ra s' \times
    s'') \eps^2 + \landau(\eps^3), N + \landau(\eps^2)\ra}
    {9 \|s'\|^4 (\kappa + \landau(\eps^2)) \|2 \eps s' + \landau(\eps^3)\|^2}
    \\
    &\overset{(\S)}{=}
    \frac{\la\|s'\|^2 s' \times s''' \eps^2, N \ra + \landau(\eps^3)}
    {\|s'\|^6 \kappa \eps^2 + \landau(\eps^2)}
    =
    \frac{\la s' \times s''', N \ra} {\|s'\|^4 \kappa} + \landau(\eps^2)
    \overset{\eqref{eq:tau}}{=}
    \tau + \landau(\eps^2),
  \end{align*}
  where we used 
  \begin{align*}
    \|\gamma_0 - \gamma_1\| 
    &= 
    \|s(-\eps) - s(\eps)\| 
    = 
    \|s - \eps s' + \frac{\eps^2}{2} + \landau(\eps^3) 
    - (s + \eps s' + \frac{\eps^2}{2} + \landau(\eps^3))\| 
    \\
    &=
    \|2 \eps s' + \landau(\eps^3)\|
  \end{align*}
  at $(*)$ and $\la s' \times s'', N\ra = 0$ at $(\S)$.
\end{proof}

\begin{rem}
  We have now a curvature and torsion for a discrete space curve as well
  as an osculating sphere and osculating circle. In the setting of smooth
  curves the oriented distance between the center of the osculating circle
  and the osculating sphere is
  $$
    \frac{\kappa'}{\kappa^2 \tau} 
  $$
  as follows immediately from the formula for the center of the osculating
  sphere, Eq.~\eqref{eq:osculatingsphere}.  We can therefore define
  a discrete version of $\kappa'$ as that value that fulfills the equation
  above by replacing smooth notions by their discrete counterparts.
\end{rem}

\section{Geometric properties}
\label{sec:geometric}

\begin{figure}[t]
  \begin{minipage}[t]{.56\textwidth}
    \begin{overpic}[width=.49\textwidth]{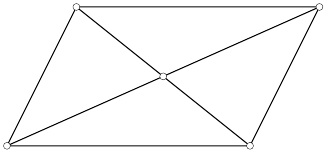}
      \lput(0,0){$a$}
      \lput(20,41){$c$}
      \rput(80,0){$b$}
      \rput(100,41){$d$}
      \cput(48,14){$f$}
    \end{overpic}
    \hfill
    \begin{overpic}[width=.46\textwidth]{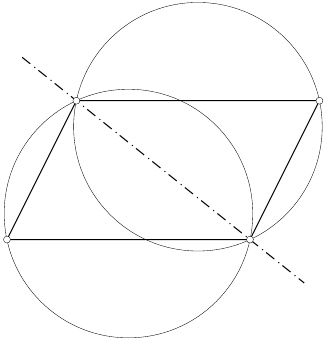}
      \lput(0,29){$a$}
      \lput(18,69){$c$}
      \rput(73,21){$b$}
      \rput(96,67){$d$}
    \end{overpic}
  \end{minipage}
  \hfill
  \begin{minipage}[b]{.42\textwidth}
  \caption{\emph{Left:} Any quadrilateral $a, b, c, d$ is M{\"o}bius
    equivalent to a parallelogram with $b - a = d - c$. In case of such a
    parallelogram we have $\crr(c, a, b, f) = -\sqrt{\crr(c, a, b, d)}$.
    \emph{Right:} The circumcircles of $a, b, c$ and $b, d, c$ are
    congruent. One of their two bisector circles is a straight line, the
    diagonal $cb$.
    }
  \label{fig:parallelogram}
  \end{minipage}
\end{figure}

Any quadrilateral is M{\"o}bius equivalent to a parallelogram, and even more
special, it is M{\"o}bius equivalent to a parallelogram $a, b, c, d$ with
$a = 0$ and $b = 1$ and $b - a = d - c$. See
Figure~\ref{fig:parallelogram} (left). Its cross\dash ratio is 
$$
\crr(c, a, b, d)
=
c^2.
$$
If $f$ denotes the intersection point of the diagonals then we obtain
$$
\crr(c, a, b, f)
=
-c,
$$
and therefore
$$
\crr(c, a, b, f)
=
-
\sqrt{\crr(c, a, b, d)}.
$$
This together with Lemma~\ref{lem:crequation} immediately implies the
following Lemma.
\begin{lem}
  Let $a, b, c, d$ be a parallelogram with $b - a = d - c$.
  Then the insertion point $f(a, b, c, d)$ corresponds to the intersection
  point of the diagonals.
\end{lem}

Furthermore, in the case of a parallelogram the circumcircles of $a, b, c$
and $b, d, c$ are congruent.  Therefore, one of their two bisector circles
is the straight line containing the diagonal $bc$ (see
Figure~\ref{fig:parallelogram} right). The same holds for the other
pair of circumcircles $a, b, d$ and $a, d, c$. Since M{\"o}bius
transformations do not change the intersection angles of curves we obtain
the following lemma.
\begin{lem}
  The insertion point $f(a, b, c, d)$ is one of the intersection points of
  the bisector circles of the pairs of circumcircles mentioned above. 
\end{lem}

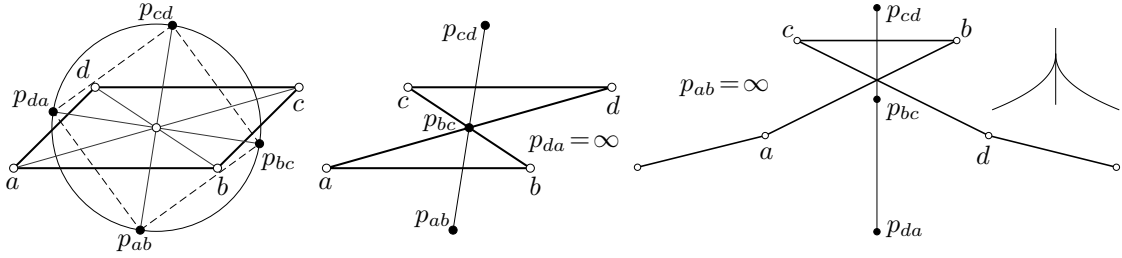
\begin{figure}[t]
  \begin{overpic}[width=.27\textwidth]{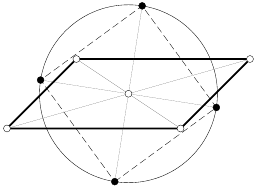}
    \rput(0,16){$a$}
    \rput(70,14){$b$}
    \lput(99,42){$c$}
    \lput(28,53){$d$}
    \lput(49,-3){$p_{ab}$}
    \rput(85,24){$p_{bc}$}
    \lput(56,74){$p_{cd}$}
    \lput(15,45){$p_{da}$}
  \end{overpic}
  \hfill
  \begin{overpic}[width=.27\textwidth]{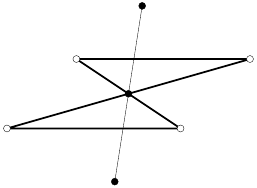}
    \rput(0,16){$a$}
    \rput(70,14){$b$}
    \lput(30,43){$c$}
    \lput(100,41){$d$}
    \lput(44,5){$p_{ab}$}
    \lput(46,37){$p_{bc}$}
    \lput(53,67){$p_{cd}$}
    \lput(100,30){$p_{da}\! =\! \infty$}
  \end{overpic}
  \hfill
  \begin{overpic}[width=.44\textwidth]{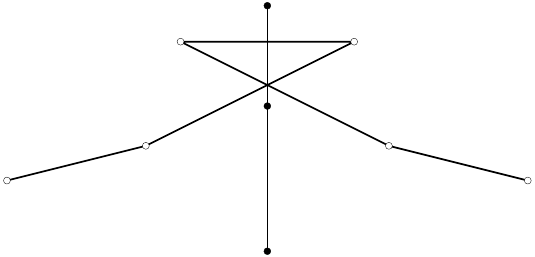}
    \lput(29,16){$a$}
    \rput(67,41){$b$}
    \lput(33,41){$c$}
    \lput(73,15){$d$}
    \rput(52,1){$p_{da}$}
    \rput(52,25){$p_{bc}$}
    \rput(52,45){$p_{cd}$}
    \rput(10,30){$p_{ab}\! =\! \infty$}
    \lput(100,25){\includegraphics[width=.12\textwidth]{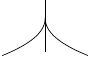}}
  \end{overpic}
  \caption{Special cases of four points together with their curvature
    circle:
    \emph{Left:} The points $p_{ab}, p_{bc}, p_{cd}, p_{da}$ form a
    square.
    \emph{Center:} The curvature circle degenerates to a straight line.
    \emph{Right:} Symmetric curve with a ``loop'' that can be interpreted
    as discrete cusp of the curve. Consistently, the curvature circle
    degenerates to a straight line.
  }
  \label{fig:spezial}
\end{figure}

Consequently, the discrete curvature circle can be constructed with
compass and ruler.  In the following lemma we mention three special cases:

\begin{lem}
  \begin{enumerate}\num
    \item\label{item:speciali} 
      Let $a, b, c, d$ be a parallelogram with $b - a = c - d$.
      Then the four points $p_{ab}, p_{bc}, p_{cd}, p_{da}$ form a square
      (see Figure~\ref{fig:spezial} left).
    \item\label{item:specialii} 
      Let $a, b, c, d$ be a parallelogram with $b - a = d - c$.
      Then $p_{da} = \infty$ and the curvature circle degenerates to a
      straight line (see Figure~\ref{fig:spezial} center).
    \item\label{item:specialiii} 
      Let $a, b, c, d$ be symmetric as in Figure~\ref{fig:spezial}
      (right).  Then $p_{ab} = \infty$ and the curvature circle
      degenerates to a straight line. Thus, this arrangement of points can
      be seen as a discrete analogue of a cusp on a curve.
  \end{enumerate}
\end{lem}
\begin{proof}
  \underline{ad \ref{item:speciali}:}
  The rotational symmetry by an angle of $\pi$ of the parallelogram
  implies that the points $p_{ab}$ and $p_{cd}$ are opposite of the center
  of rotation as well as $p_{bc}$ and $p_{da}$. A quadrilateral with this
  property and with a cross\dash ratio of $-1$ (Theorem~\ref{thm:circle})
  must be a square.

  \underline{ad \ref{item:specialii} and \ref{item:specialiii}:}
  It follows from simple computations that $p_{da}$ and $p_{ab}$,
  respectively, vanish to $\infty$. Circles containing this point are
  straight lines.
\end{proof}

\section{Experimental results}
\label{sec:num}

We conducted convergence tests which empirically verify our claims. For this,
we used the following seven curves (see Figure~\ref{fig:curves-display} for
their depiction):

\begin{figure}[h]
  \resizebox{\textwidth}{!}{%
  \includegraphics[height=10mm]{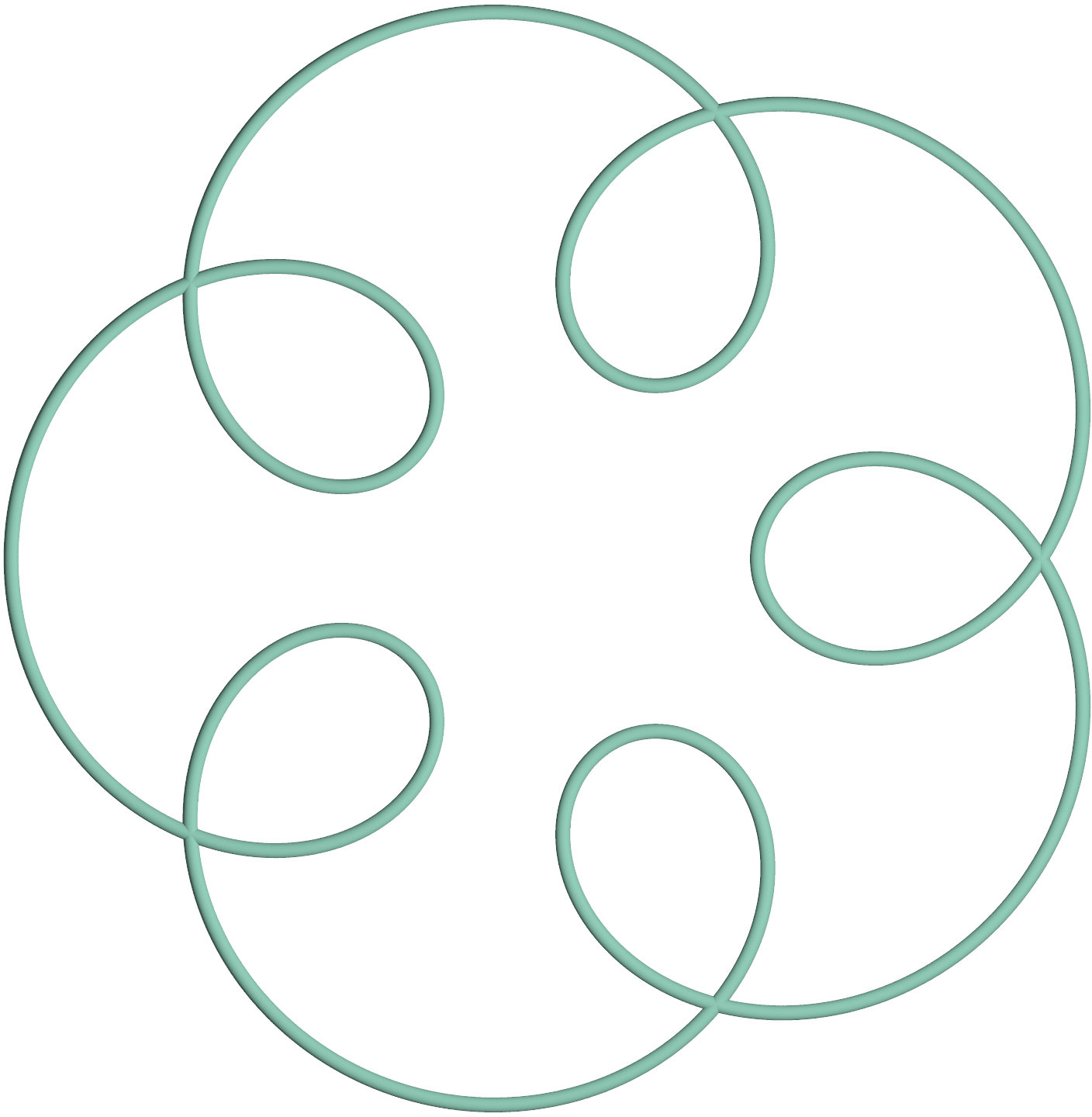}\,%
  \includegraphics[height=10mm]{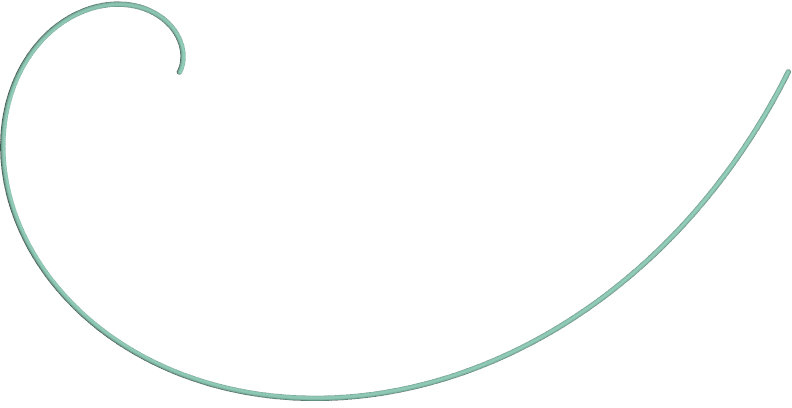}\,%
  \includegraphics[height=10mm]{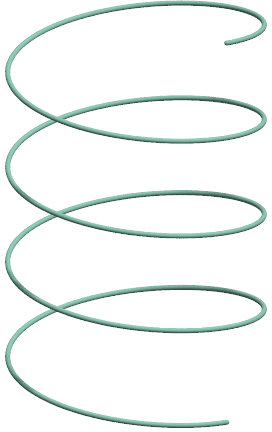}\,%
  \includegraphics[height=10mm]{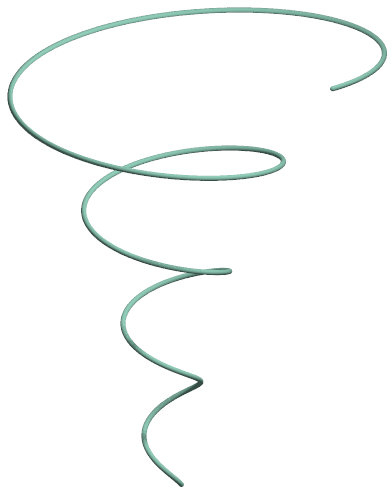}\,%
  \includegraphics[height=10mm]{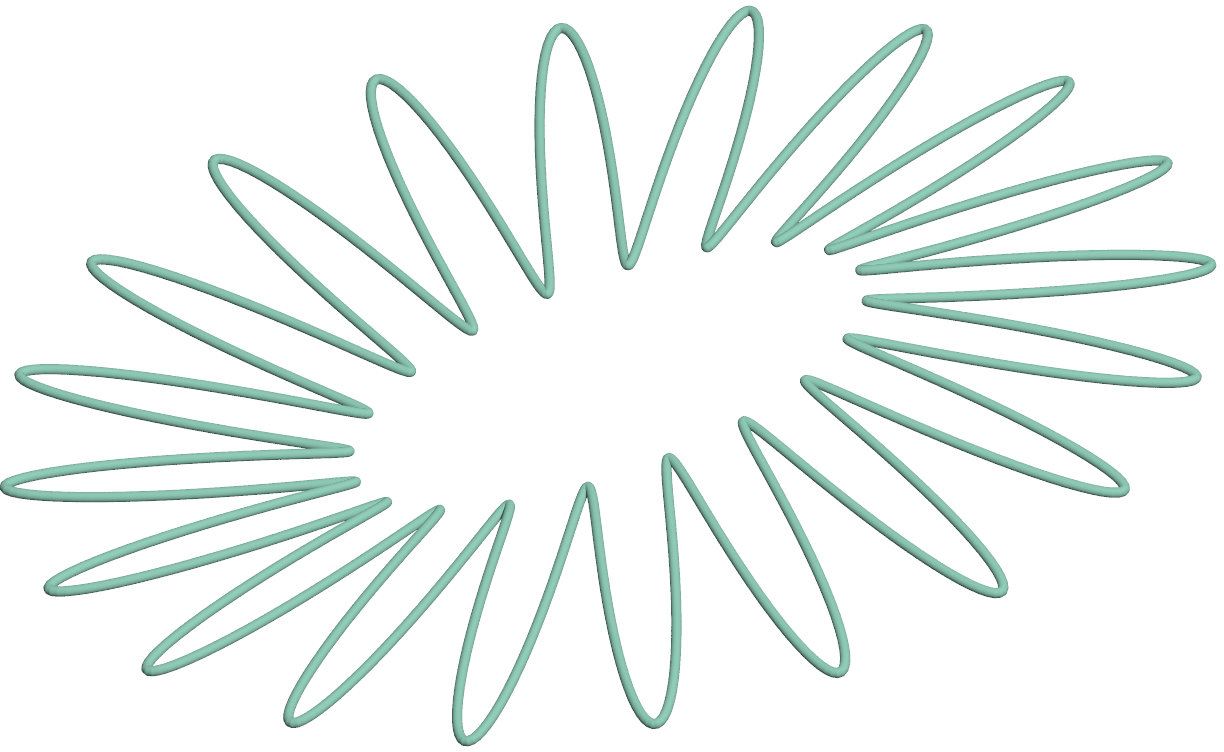}\,%
  \includegraphics[height=10mm]{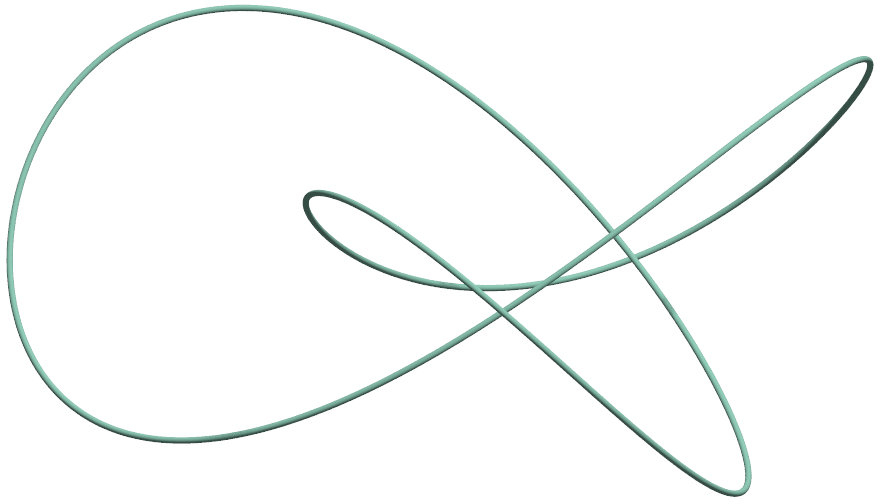}\,%
  \includegraphics[height=10mm]{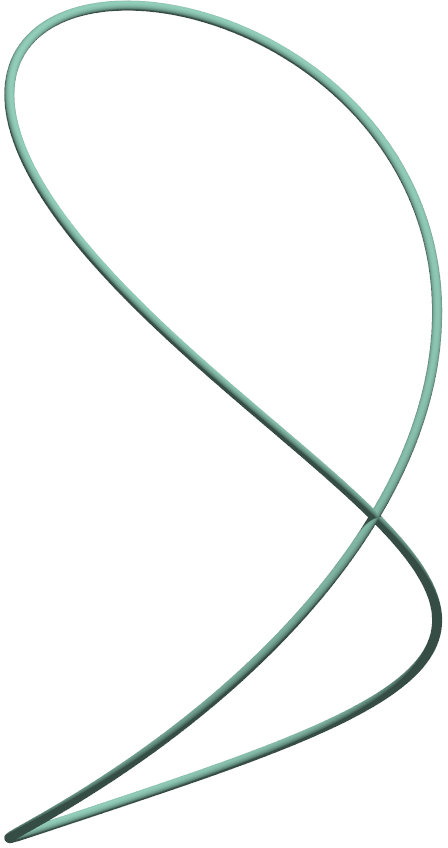}\,%
  }
  \caption{Illustration of the list of curves used for our numerical
  convergence verification.}
  \label{fig:curves-display}
\end{figure}

\begin{enumerate}\num
  \item The epitrochoid 
    $c_1(t) = (6 \cos(t) - 3 \cos(6t),  6 \sin(t) - 3 \sin(6t))$ 
    (the curve is planar).
  \item A planar logarithmic spiral 
    $c_2(t) = e^{at} (\cos(t), \sin(t))$,
    where we use $a = 0.5$.
  \item A helix 
    $c_3(t) = (\cos(at), \sin(at), b t)$ 
    where we use $a=4$ and $b=0.5$.
  \item A helical spiral 
    $c_4(t) = (e^{at} \cos(4t), e^{at} (\sin(4t), bt)$ 
    where we use $a=0.4$ and $b=4$.
  \item A toroidal ``coil'' 
    $c_5(t) = ((a + \sin(bt)) \cos(t), (a + \sin(bt)) \sin(t),
    \cos(bt))$ 
    where we use $a=2.5$ and $b=20$.
  \item The trefoil knot 
    $c_6(t) = (\sin(t) + 2 \sin(2t), \cos(t) - 2 \cos(2t), -\sin(3t))$.
  \item Viviani's curve~\cite{gray:2006} 
    $c_7(t) = (a (1 + \cos(2t)), a \sin(2t), 2 a \sin(t))$
    with $a=5$.
\end{enumerate}

For all examples, we used $t \in [0, 2\pi]$. For simplicity, we
assumed all curves are open, and disregarded boundaries; that is we do not
compute edge midpoint and consequent quantities for edges adjacent to
boundary vertices. The curves are not assumed to be arc\dash length
parametrized.

For any given resolution step $\eps$, we created a discrete curve by
sampling every curve $c_i(t)$, as explained in
Theorem~\ref{thm:mainspacial}. Then, we compute the discrete curvature
$\kappa$, the discrete torsion $\tau$, and the discrete Frenet frame
$\left\{T,N,B\right\}$ for every midedge point. We measure the
approximation error to the corresponding quantities of the smooth curve at
the sampled points by the $l^{\infty}$ norm. This produces the maximum
absolute deviation of every discrete quantity from the ground truth. In
case of vector quantities (like the Frenet frame), we do so per component.
We use $\eps =  0.1 \times 1.1^l$, where $l \in \mathbb{N}$ runs
between $0$ and $-15$ in steps of $-1$, which creates gradual refinement.
To measure convergence rate, we perform linear regression on the
logarithmic scale of $\eps$ vs.\ $l^{\infty}$ error per curve. The
graphs of errors can be seen in Figure~\ref{fig:error-graphs}, and the
convergence rates are in Table~\ref{table:converges-rates}. It is evident
that we are able to reproduce the quadratic convergence rates that we
prove in this paper. Note that we do not measure torsion for $c_1(t)$ and
$c_2(t)$ as they are planar. Another outlier is the normal error for
$c_3(t)$ which is already initially very low (due to the high regularity
of the helix), and thus we only see the effect of numerical noise.

\begin{figure}
\includegraphics[width=0.32\textwidth]{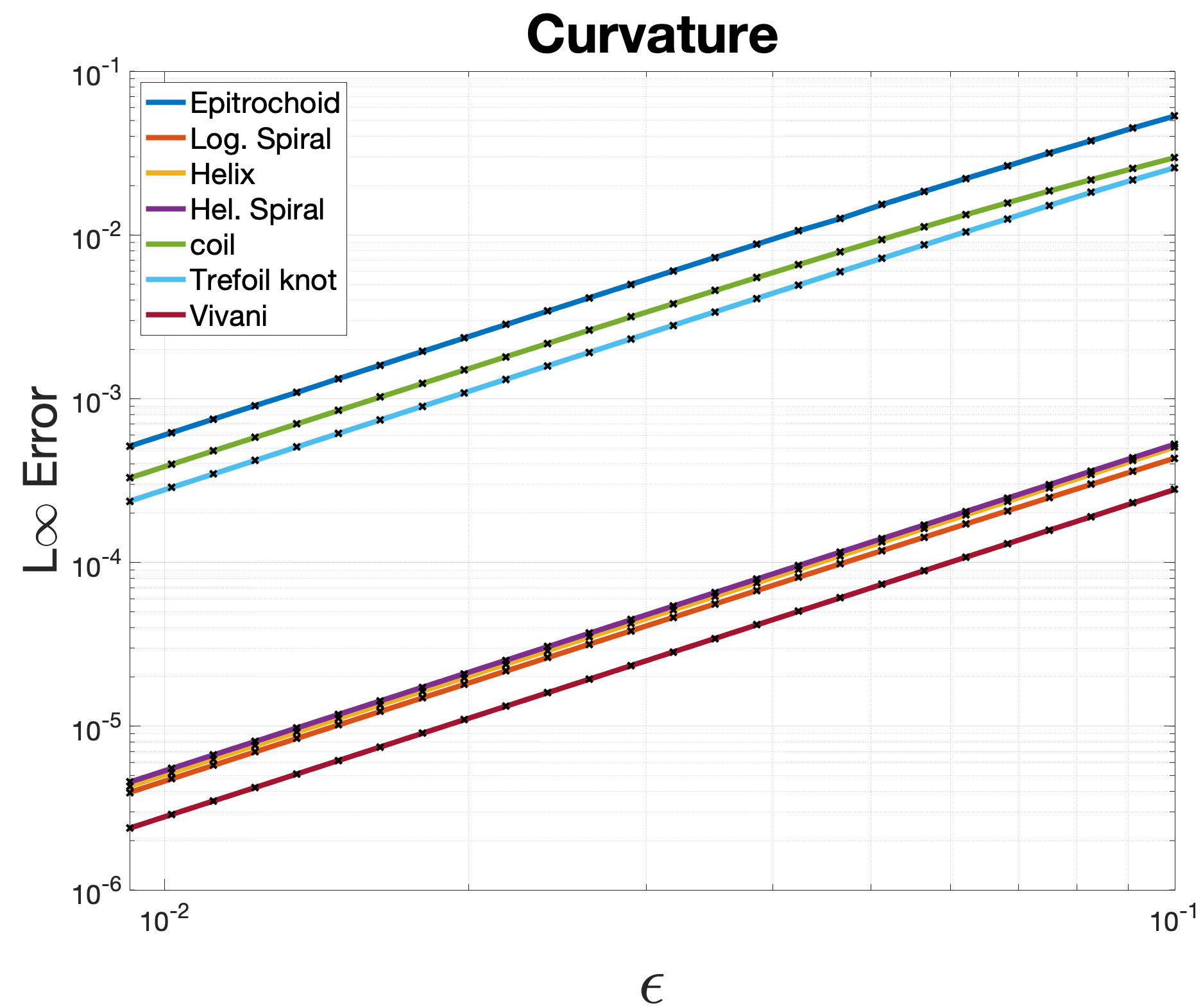}
\includegraphics[width=0.32\textwidth]{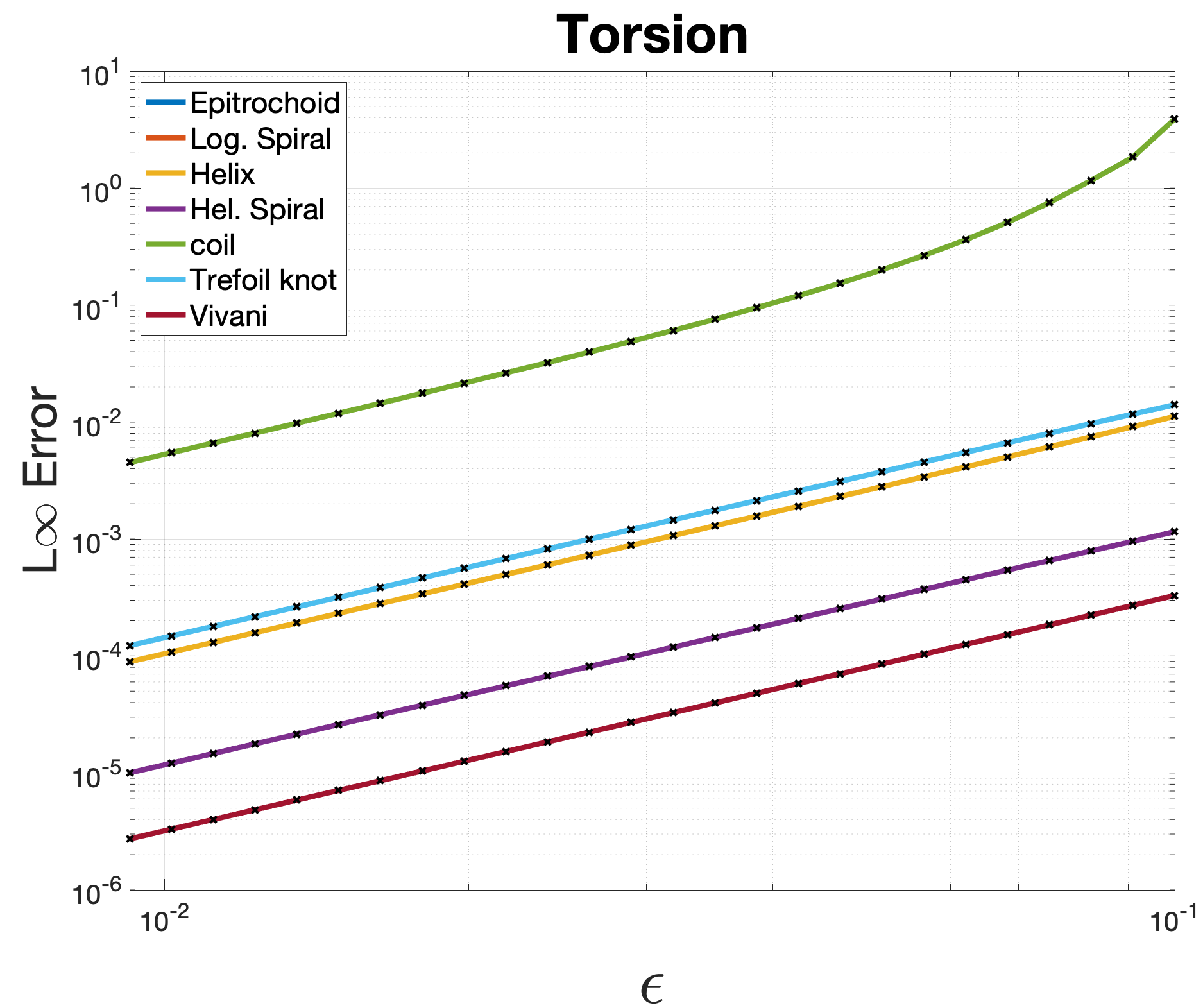}
\includegraphics[width=0.32\textwidth]{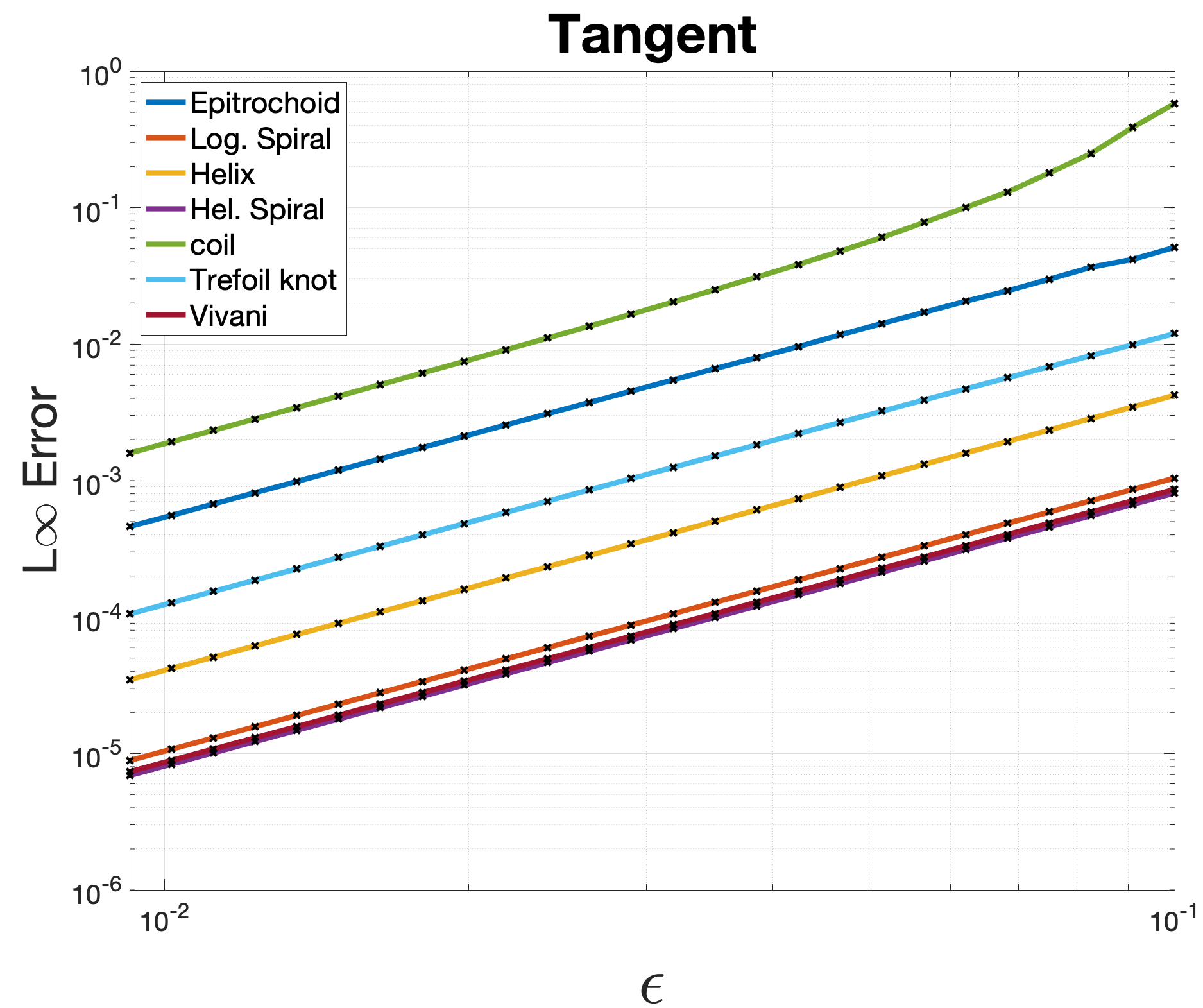}
\includegraphics[width=0.32\textwidth]{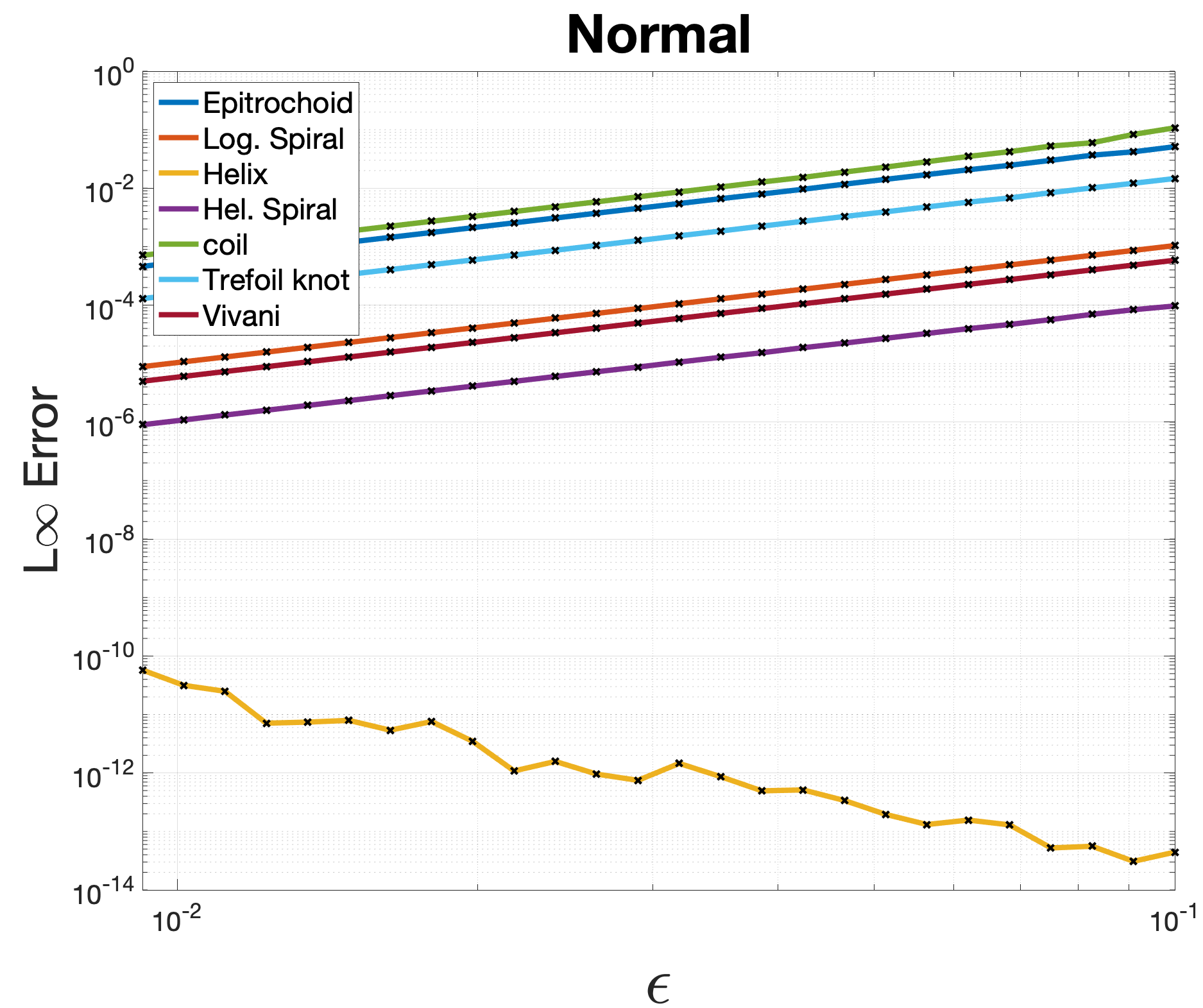}
\includegraphics[width=0.32\textwidth]{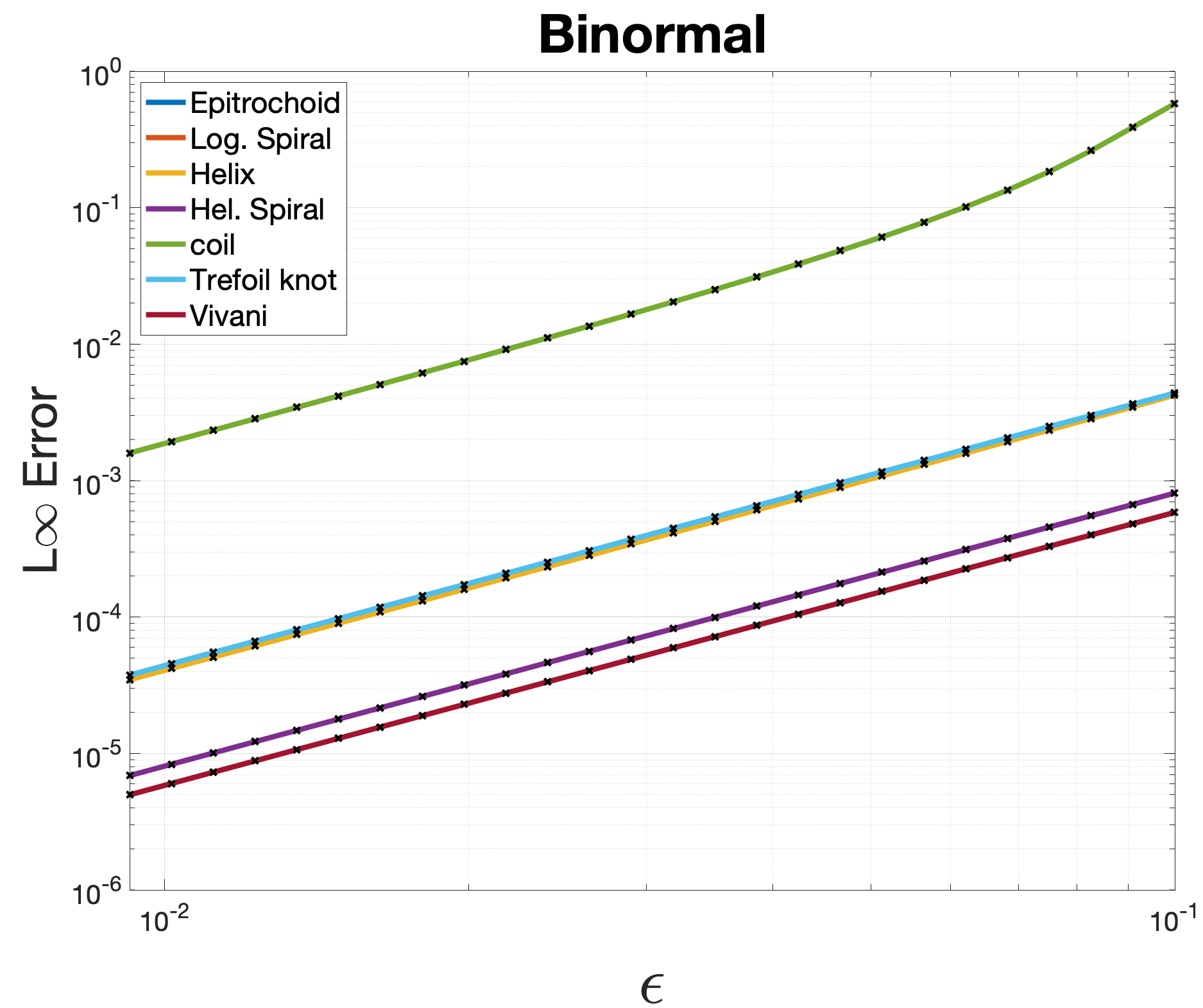}
\caption{$l^{\infty}$ errors vs. sampling step $\eps$.}
\label{fig:error-graphs}
\end{figure}

\begin{table}

\begin{tabular}{| c | c | c | c | c| c |}
\hline
Curve & $\kappa$ & $\tau$ & $T$ & $N$ & $B$ \\ 
\hline
  (i) & 1.9589	& -		& 1.9858	& 1.9858	&   - 		\\     
  (ii) & 1.9745	& -		&  2.0005	&2.0005	&-		\\
  (iii) &   2.0010 	& 2.0212	& 2.0122	& -		&2.0122	\\
  (iv) & 1.9947	& 1.9934	& 2.0003	&1.9742	&2.0002	\\
  (v) & 1.9096 	& 2.5707	& 2.3352	&2.0647	&2.3414	\\
  (vi) & 1.9772  	& 1.9936	& 1.9888	&1.9864	&1.9980	\\
  (vii) & 1.9986	& 2.0102	& 2.0000	&1.9996	&2.0002	\\
\hline
\end{tabular}
\caption{Error convergence rates with refinement. Note that there is no torsion or non-trivial binormal for the planar curves $c_1(t)$ and $c_2(t)$.}
\label{table:converges-rates}
\end{table}

\section*{Acknowledgements}

The first author gratefully acknowledges the support of the Austrian
Science Fund (FWF) through projects P~29981 and I~4868.

\bibliographystyle{plain}
\bibliography{moebiuscurvaturecircle}

\end{document}